\documentclass[12pt]{amsart}

\usepackage[latin1]{inputenc}
\usepackage{latexsym,amsmath,amssymb,amscd,amsfonts,amsthm, mathrsfs}
\usepackage{array}
\usepackage{geometry}

\usepackage{color}

\usepackage[active]{srcltx}

\newtheorem{theorem}{Theorem}[section]
\newtheorem{lemma}[theorem]{Lemma}
\newtheorem{proposition}[theorem]{Proposition}
\newtheorem{corollary}[theorem]{Corollary}
\newtheorem{definition}[theorem]{Definition}
\newtheorem{example}[theorem]{\sc Example}
\newtheorem{remark}[theorem]{Remark}

\newcommand {\C}    {\mathbb{C}}
\newcommand {\D}    {\mathbb{D}}

\newcommand {\R}    {\mathbb{R}}

\renewcommand{\epsilon}{\varepsilon}

\allowdisplaybreaks
\begin{document}
\date{\today}
%\dedicatory{}


\title[UC functions and quantization]{Uniform continuity and quantization \\ on bounded symmetric domains}

\author[W. Bauer]{W. Bauer}
\address{Institut für Analysis \endgraf
Welfengarten 1, 30167 Hannover, Germany \endgraf}
\email{bauer@math.uni-hannover.de}

\author[R. Hagger]{R. Hagger}
\address{Institut für Analysis \endgraf
Welfengarten 1, 30167 Hannover, Germany \endgraf}
\email{raffael.hagger@math.uni-hannover.de}

\author[N. Vasilevski]{N. Vasilevski}
\address{Departamento de Matem\'{a}ticas, CINVESTAV \endgraf
Apartado Postal 14-740, 07000, M\'{e}xico, D.F., M\'{e}xico \endgraf}
\email{nvasilev@math.cinvestav.mx}

\thanks{The first and third author acknowledge support through DFG (Deutsche Forschungsgemeinschaft), \\BA 3793/4-1.}

\begin{abstract}
We consider Toeplitz operators $T_f^{\lambda}$ with symbol $f$ acting on the standard weighted Bergman spaces over a bounded symmetric domain $\Omega\subset \mathbb{C}^n$. 
Here $\lambda > \mbox{\it genus}-1$ is the weight parameter. The classical asymptotic relation for the semi-commutator
\begin{equation}\tag{$*$}
\lim_{\lambda \rightarrow \infty} \big{\|}T_f^{\lambda} T_g^{\lambda} -T_{fg}^{\lambda} \big{\|}=0, \hspace{2ex} \mbox{\it with} \hspace{2ex} f,g \in C(\overline{\mathbb{B}^n}), 
\end{equation}
where $\Omega=\mathbb{B}^n$ denotes the complex unit ball, is extended to larger classes of bounded and unbounded operator symbol-functions and to more general domains. We deal with 
operator symbols that generically are neither continuous inside $\Omega$ (Section \ref{Section_VMO}) nor admit a continuous extension to the boundary (Section \ref{Section_uniformly_continuous} and \ref{Section_VMO}). 
Let $\beta$ denote the Bergman metric distance function on $\Omega$. We prove that ($*$) remains true for $f$ and $g$ in the space ${\rm UC}(\Omega)$ of all $\beta$-uniformly continuous 
functions on $\Omega$. Note that this space contains also unbounded functions. In case of the complex unit ball $\Omega=\mathbb{B}^n \subset \mathbb{C}^n$ we show that ($*$) 
holds true for bounded symbols in ${\rm VMO}(\mathbb{B}^n)$, where the vanishing oscillation inside $\mathbb{B}^n$ is measured with respect to $\beta$. At the same time ($*$) fails 
for generic bounded measurable symbols.  We construct a corresponding counterexample using oscillating symbols that are continuous outside of a single point in $\Omega$.
\end{abstract}
\maketitle
%%%%%%%%%%%%%%%%%%%%%%%%%%%%%%%%%%%%%%%%%%%%%%%%%%%%%%%%%%%%%%%%%%%%%%%%%
\section{Introduction}
\label{section_introduction}
\setcounter{equation}{0}
%%%%%%%%%%%%%%%%%%%%%%%%%%%%%%%%%%%%%%%%%%%%%%%%%%%%%%%%%%%%%%%%%%%%%%%%%%%
Let $\Omega\subset \mathbb{C}^n$ be a bounded symmetric domain (shortly BSD) and consider a (suitable) algebra of functions on $\Omega$. It is a classical scheme in 
{\it deformation quantization} to construct an associated family of non-commutative algebras $\mathcal{A}_{\lambda}$ that depend on a deformation parameter $\lambda$, and such that in the semi-classical 
limit (i.e.~when the Planck constant $\hbar \sim \frac{1}{\lambda}$ tends to zero)  $\mathcal{A}_{\lambda}$  should approach in some sense the above commutative algebra of functions, cf. \cite{Be1,Be}. 
A classical method for constructing a deformation quantization of symmetric spaces uses Toeplitz operators as quantum counterparts of the functions we start with. Such operators are  defined 
on the standard weighted Bergman spaces over $\Omega$, cf.  \cite{BC-1,Bro,BLU,E2,KL,Rie} and the weight parameter explicitly appears in the density function of the (Lebesgue) measure restricted 
to $\Omega$. Essential relations that one needs to prove (cf. \cite{Rie}) are the norm convergence ($*$) and (assuming some smoothness of the symbols) the second order asymptotic 
\begin{equation}\tag{$**$}
\big{\|} [T^{\lambda}_f, T_g^{\lambda}]-\frac{i}{\lambda} T_{\{f,g\}}^{\lambda} \big{\|}=O(\lambda^{-2}) \hspace{5ex} \mbox{\it as} \quad \lambda \rightarrow \infty. 
\end{equation}
Here $[\cdot , \cdot]$ denotes the commutator of operators and $\{ \cdot, \cdot\}$ is the Poisson bracket which is associated to a symplectic form induced by the Bergman metric tensor. 
\vspace{1mm}\par
Recall that a quantization via Toeplitz operators acting on the Bergman space was first introduced by F. Berezin \cite{Be,Be1,Be2} for the case of the unit disk $\mathbb{D}$ in the complex plane and, 
more generally, for BSDs $\Omega \subset \mathbb{C}^n$. An a bit different approach to quantization for the unit disc $\mathbb{D}$ has been considered by Klimek and Lesniewski 
in \cite{KL} and subsequently was generalized to arbitrary BSDs by Borthwick, Lesniewski and Upmeier in \cite{BLU}. Deformation estimates for Berezin-Toeplitz quantization on the Euclidean $n$-space 
$\Omega = \mathbb{C}^n$ equipped with a family of Gaussian measures were obtained in \cite{Bro,C0}. In this non-compact setting the proofs are based on the relation between Toeplitz operators and 
pseudo-differential operators in Weyl-quantization. In particular, the required norm estimates are a consequence of the Calderon-Vaillancourt theorem. For $\Omega$ being a compact Kähler 
manifold, the above asymptotic relations have been obtained by Bordemann, Meinrenken and Schlichenmaier \cite{BMS} (see also \cite{MM}). An analysis of the semi-classical limit for smoothly 
bounded strictly pseudoconvex domains in $\mathbb{C}^n$ can be found in \cite{E2}. We mention as well that a family of associative star products in deformation quantization can be constructed on the base of 
$(*)$ and $(**)$, cf. \cite{E2}. 
\vspace{1mm}\par 
The above mentioned results typically require certain regularity of the operator symbols and their controlled behavior close to the boundary of the domain (or at infinity). More precisely, in \cite[Theorem 2.2]{BLU} 
the relation $(*)$ is proved  assuming that $f$ and $g$ are bounded continuous functions and $g$ has compact support in $\Omega$.  In the special case of $\Omega = \mathbb{B}^n$ we may as well 
apply \cite[Theorem 3]{E2}, which assumes that $f$ and $g$ are smooth up to the boundary of $\Omega$.  If one is only interested in $(*)$, this assumption can be relaxed to $f,g \in C(\overline{\Omega})$ by a 
simple approximation argument. 
\vspace{1mm}\par 
The aim of the present paper is to extend $(*)$ for symbols $f,g$ in larger algebras of bounded (and unbounded) functions in BSDs $\Omega$. We show that $(*)$ holds true if $f$ and $g$ are bounded and 
uniformly continuous in $\Omega$ with respect to the Bergman metric distance $\beta$. Note that in general such functions do not extend continuously to the boundary $\partial \Omega$. Moreover, 
we can even drop the boundedness assumption and obtain $(*)$ for (unbounded) Toeplitz operators with $\beta$-uniformly continuous symbols (cf. Theorem 
\ref{Theorem_behaviour_semi_commutator_uniformly_bounded_function}). At the same time Example \ref{counterexample_fast_oscillating_symbols} shows that $(*)$ may fail if we drop the continuity 
assumption even in one single point inside $\Omega$. In the last section of the paper we deal only with the complex unit ball $\Omega = \mathbb{B}^n$. We emphasize that a controlled oscillation of 
bounded symbols $f$ and $g$ inside $\Omega$ implies $(*)$. To be precise, assuming that $f$ or $g$ belongs to the space $\textup{VMO}(\mathbb{B}^n)$ of bounded functions having 
vanishing oscillation with respect to $\beta$ is sufficient for $(*)$. Our proofs use a refinement of the norm estimates for  Hankel operators in \cite{BBCZ} and an asymptotic analysis of the $\textup{BMO}^{\lambda}$-semi-norms 
of $\beta$-uniformly continuous functions (cf. Proposition \ref{estimate_norm_Hankel_operator} and \ref{Lemma_limit_BMO_lambda_UC}).  We remark that by different methods similar (but slightly weaker) results 
for the Fock space case (i.e.~$\Omega= \mathbb{C}^n$ equipped with a family of Gaussian measures) have been obtained recently in \cite{BC-1}. 
\vspace{1mm}
\par 
One of our motivations for considering this problem stems from the representation theory of $C^*$-algebras generated by Toeplitz operators (cf. \cite{BV}). In fact, in this paper a family of irreducible 
representations has been constructed under certain assumptions which include $(*)$. We expect our analysis to be useful for the study of Toeplitz $C^*$-algebras with generating operators having 
symbols in (suitable) classes of functions that not necessarily admit continuous boundary values. Further details shall be presented in a forthcoming work. 
\vspace{1mm}
\par 
In Section \ref{section_preliminaries} we fix the notation and present some standard material on BSDs, Bergman spaces and Toeplitz operators. In particular, we show that finite products of (i.g.~unbounded) Toeplitz 
operators with $\beta$-uniformly continuous symbols are well-defined on a common dense domain. 
We start Section \ref{Section_uniformly_continuous} with some technical estimates and use them to derive a norm estimate for Hankel operators, which is needed in the proof of our main result (Theorem \ref{Theorem_behaviour_semi_commutator_uniformly_bounded_function}). The compactness of semi-commutators are discussed and we present the above mentioned counterexample. 
%\par 
Finally, in Section \ref{Section_VMO} we prove $(*)$ in the case of bounded symbols having vanishing oscillation inside $\Omega= \mathbb{B}^n$. 
%%%%%%%%%%%%%%%%%%%%%%%%%%%%%%%%%%%%%%%%%%%%%%%%%%%%%%%%%%%%%%%%%%%%%%%%%
\section{Preliminaries}
\label{section_preliminaries}
\setcounter{equation}{0}
%%%%%%%%%%%%%%%%%%%%%%%%%%%%%%%%%%%%%%%%%%%%%%%%%%%%%%%%%%%%%%%%%%%%%%%%%%%
%%%%%%%%%%%%%%%%%%%%%%%%%%%%%%%%%%%%%%%%%%%%%%%%%%%%%%%%%%%%%%%%%%%%%%%%%%%%%%%%%%%%%%%%%%%
Throughout the paper we consider a BSD $\Omega \subset \mathbb{C}^n$ in its Harish-Chandra realization \cite{BBCZ,Cart,E2,FK,T}. In particular, $\Omega$ contains the origin and is 
convex and circular. We write $G= \textup{Aut}_0(\Omega)$ for the connected component of the automorphism group of $\Omega$ which contains the identity. 
By $K$ we denote the (maximal) subgroup of $G$ that stabilizes the origin.

As is well-known, each $k \in K$ extends to a linear mapping on $\mathbb{C}^n$ \cite{Cart}. If $r$ denotes the {\it rank} of $\Omega$, then there is a set 
$\{ f_1, \cdots, f_r\} \subset \mathbb{C}^n$ ({\it Jordan frame}) of $\mathbb{R}$-linear independent vectors such that 
\begin{equation}\label{decomposition_Omega}
\Omega= \Big{\{} z \in \mathbb{C} \: : \: z=k \sum_{j=1}^r t_j f_j, \: k \in K, \; 1 > t_1 \geq t_2 \geq \cdots \geq t_r \geq 0 \Big{\}}. 
\end{equation}
The sum-representation of $z\in \Omega$ in (\ref{decomposition_Omega}) is called {\it polar decomposition} and, assuming the above ordering, the numbers $t_j$ are uniquely determined ($k$ is not i.g.). There is a polynomial 
({\it Jordan triple determinant})
\begin{equation*}
h: \mathbb{C}^n \times \mathbb{C}^n \rightarrow \mathbb{C}
\end{equation*}
holomorphic in $z$ and anti-holomorphic in $w$ which restricted to the diagonal fulfills
\begin{equation}\label{GL_expansion_on_the_diagonal}
h(z,z)= \prod_{j=1}^r \big{(} 1-t_j^2\big{)}. 
\end{equation}
Moreover, $h$ is conjugate symmetric, i.e.~$h(z,w)=\overline{h(w,z)}$ and invariant under the action of $K$, i.e.~for all $z,w \in \Omega$ and all $k \in K$ 
\begin{equation}\label{invariance_h}
h(kz,kw)= h(z,w).% \hspace{5ex} \forall \: k \in K. 
\end{equation}
\par 
Let $p$ be the {\it genus} of $\Omega$ (see \cite{E2} for the definition) and $\lambda > p-1$. Consider the following weighted measure on $\Omega$: 
\begin{equation*}
dv_{\lambda}(z)= c_{\lambda} h(z,z)^{\lambda-p} dv(z),
\end{equation*}
where $dv$ denotes the normed to one Lebesgue measure on $\Omega$ and $c_{\nu}>0$ is a normalizing constant such that $v_{\lambda}(\Omega) = 1$, i.e. $c_p=1$. An explicit expression of 
$c_{\lambda}$ can be found in \cite{FK}. 
\vspace{1ex}\par 
We write $\mathcal{A}_{\lambda}^2(\Omega)$ for the weighted Bergman space of holomorphic functions in $L^2(\Omega, dv_{\lambda})$. The norm and inner product on these spaces are denoted by $\|\cdot \|_{\lambda}$ 
and $\langle \cdot , \cdot \rangle_{\lambda}$, respectively. The following result is well-known \cite{E3,FK}: 
\begin{lemma}\label{Lemma_JTD}
The Bergman space $\mathcal{A}^2_{\lambda}(\Omega)$ is a reproducing kernel Hilbert space and the kernel can be expressed in terms of the Jordan 
triple determinant: 
\begin{equation}\label{Formula_weighted_Bergman_kernel}
K_{\lambda}(z,w)= h(z,w)^{-\lambda}, \hspace{4ex} \mbox{\it where} \hspace{4ex} (z,w) \in \Omega \times \Omega. 
\end{equation}
\end{lemma}
We denote by $\beta_{\lambda}(\cdot, \cdot)$ the Bergman metric on $\Omega$ with respect to the weighted Bergman space $\mathcal{A}_{\lambda}^2(\Omega)$. 
More precisely, $\beta_{\lambda}$ is the metric distance function induced by the infinitesimal Bergman metric on $\Omega$ with metric tensor: 
\begin{equation*}
\big{(}g_{ij}^{\lambda}(z)\big{)}_{i,j}= \Big{(} \frac{\partial^2}{\partial z_i \partial \overline{z}_j} \log K_{\lambda}(z,z) \Big{)}_{i,j} \in \mathbb{C}^{n \times n}, 
\end{equation*}
where $K_{\lambda}$ denotes the reproducing kernel function as defined above. Then we have 
\begin{equation}\label{weighted_Bergman_metric}
\beta_{\lambda}(z,w)= \sqrt{\frac{\lambda}{p}} \beta(z,w), \hspace{3ex} \mbox{\it with the definition} \hspace{3ex}  \beta(z,w)=\beta_p(z,w). 
\end{equation}
%%%%%%%%%%%%%%%%%%%%%%%%%%%%%%%%%%%%%%%%%%%%%%%%%%%%%%%%%%%%%%%%%%%%%%%%%%%%%%%%%%%%%%%%%%%%%%%%%%%%%%%%
\subsection{Functions of bounded and vanishing oscillation}
With fixed $w \in \Omega$ consider the normalized reproducing kernel $k^{\lambda}_w \in \mathcal{A}_{\lambda}^2(\Omega)$
\begin{equation*}
k^{\lambda}_w(z):= K_{\lambda}(z,w) \|K_{\lambda}(\cdot, w)\|_{\lambda}^{-1}=h(z,w)^{-\lambda} h(w,w)^{\frac{\lambda}{2}}, \hspace{3ex} z \in \Omega. 
\end{equation*} 
The {\it Berezin transform} of $f \in L^1(\Omega, dv)$ (see \cite[Lemma 4.1]{BC1}) is the real analytic function on $\Omega$ defined by the integral transform 
\begin{equation}\label{Defintion_Berezin_Transform_of_a_function}
\mathcal{B}_{\lambda}(f)(z):= \int_{\Omega} f(w) |k_z^{\lambda}(w)|^2 dv_{\lambda}(w)
\end{equation}
(see (\ref{Berezin_transform_second_form}) for yet another representation of the Berezin transform). Recall that the {\it mean oscillation} of $f\in L^2(\Omega, dv)$ at $z \in \Omega$ is given by:
\begin{equation}\label{Definition_rem_mean_oscillation}
\textup{MO}^{\lambda}(f)(z):= \mathcal{B}_{\lambda}(|f|^2)(z)- |\mathcal{B}_{\lambda}(f)|^2(z)= \mathcal{B}_{\lambda} \Big{(} |f-\mathcal{B}_{\lambda}(f)(z)|^2\Big{)}(z) \geq 0.
\end{equation}
We consider the family of semi-norms
\begin{equation*}
\|f\|_{\textup{BMO}^{\lambda}}:= \sup \Big{\{} \sqrt{\textup{MO}^{\lambda}(f)(z)} \: : \: z \in \Omega \Big{\}}. 
\end{equation*}
The space of functions having {\it bounded $\lambda$-mean oscillation} is given by: 
\begin{equation*}
\textup{BMO}^{\lambda}(\Omega):= \Big{\{} f : \Omega \rightarrow \mathbb{C} \: : \: \|f\|_{\textup{BMO}^{\lambda}} < \infty \Big{\}}. 
\end{equation*}
\par 
In what follows we shortly write $\textup{BMO}(\Omega):= \textup{BMO}^p(\Omega)$ and $\textup{MO}(f):= \textup{MO}^p(f)$. 
Note that for all $\lambda > p-1$: 
$$L^{\infty}(\Omega) \subsetneq \textup{BMO}^{\lambda}(\Omega).$$ 
\par 
Let $C_0(\Omega)= \{ f \in C(\Omega) \: : \: \lim_{z \rightarrow \partial \Omega} f(z)=0\}$ be the space of all continuous functions vanishing  at the boundary $\partial \Omega$. 
\begin{definition}
A complex valued function $g$ on $\Omega$ is said to have {\it ''vanishing mean oscillation''} at $\partial \Omega$ if $\textup{MO}(g) \in C_0(\Omega)$. Put 
\begin{equation*}
\textup{VMO}_{\partial}(\Omega):= \Big{\{} g: \Omega \rightarrow \mathbb{C} \: : \: \textup{MO}(g) \in C_0(\Omega) \Big{\}}. 
\end{equation*}
\end{definition}
\par 
There is also the notion of bounded oscillation with respect to the weighted Bergman metric $\beta_{\lambda}$: 
\begin{definition}
Let $\lambda > p-1$. A continuous function $f$ is said to be of "bounded $\lambda$-oscillation" on $\Omega$ if
\begin{equation}\label{GL_local_estimate}
\|f\|_{\textup{BO}^{\lambda}}:= \sup\Big{\{} |f(z)-f(w)| \: : \: z,w \in \Omega, \; \; \beta_{\lambda}(z,w) \,  < \, 1 \Big{\}} < \infty. 
\end{equation}
Clearly, one has for $\lambda \geq \mu$: 
$$\|f \|_{\textup{BO}^{\lambda}} \leq \|f \|_{\textup{BO}^{\mu}}.$$
\par 
We say that the function $f$ has ''vanishing $\lambda$-oscillation at $\partial\Omega$'' if  $\textup{Osc}_z^{\lambda}(f) \in C_0(\Omega)$, where 
\begin{equation*}
\textup{Osc}_z^{\lambda}(f):= \sup \Big{\{} |f(z)-f(w)| \: : \: w \in \Omega, \; \; \beta_{\lambda}(z,w)\, <\, 1 \Big{\}}, \hspace{4ex} z \in \Omega,
\end{equation*}
denotes the $\lambda$-oscillation of $f$ in $z$. We write $\textup{BO}^{\lambda}(\Omega)$ and $\textup{VO}^{\lambda}_{\partial}(\Omega)$ for the functions 
having bounded and vanishing $\lambda$-oscillation, respectively. For $\lambda = p$ we omit the superscript $p$: 
\[\textup{BO}(\Omega):=\textup{BO}^p(\Omega) \hspace{3ex}\mbox{\it and } \hspace{3ex} \textup{VO}_{\partial}(\Omega):= \textup{VO}^p_{\partial}(\Omega).\]
\end{definition}
By choosing a geodesic curve between two points $z,w \in \Omega$ and using (\ref{GL_local_estimate}) we obtain a global estimate: 
\begin{lemma}\label{GL_estimate_global_BO}
Let $f \in \textup{BO}^{\lambda}(\Omega)$. Then for all $w,z \in \Omega$ we have
\begin{equation}\label{Global_estimate_BO}
| f(z)-f(w)| \leq \|f\|_{\textup{BO}^{\lambda}} \big{[}1+ \beta_{\lambda}(z,w) \big{]}.
\end{equation}
\end{lemma}
\begin{proof}
In case of $\beta_{\lambda}(z,w) < 1$ we have (\ref{Global_estimate_BO}) by definition. Otherwise we divide the geodesic curve in pieces and use the same calculation as in 
\cite[Lemma 8.2, p. 209]{Zhu_book_0}. 
\end{proof}
%%%%%%%%%%%%%%%%%%%%%%%%%%%%%%%%%%%%%%%%%%%%%%%%%%%%%%%%%%%%%%%%%%%%%%%%%%%%%%%%%%%%%%%%%%
\subsection{Toeplitz and Hankel operators} 
\label{UC_products}
 %%%%%%%%%%%%%%%%%%%%%%%%%%%%%%%%%%%%%%%%%%%%%%%%%%%%%%%%%%%%%%%%%%%%%%%%%%%%%%%%%%%%%%%%%%
Our aim of this work is to analyze the asymptotic behavior of semi-commutators of Toeplitz operators when sending the weight parameter $\lambda>p-1$ to infinity. First we fix some basic notations. 
Consider the orthogonal projection 
\[P_{\lambda}: L^2(\Omega, dv_{\lambda}) \rightarrow \mathcal{A}_{\lambda}^2(\Omega).\]
Using Lemma \ref{Lemma_JTD}, one can write $P_{\lambda}$ explicitly as
\[(P_{\lambda}f)(z) = \int_{\Omega} f(w)h(z,w)^{-\lambda} \, dv_{\lambda}(w).\]
Given a symbol $f \in L^{\infty}(\Omega)$ we introduce the {\it Toeplitz operator} $T_f^{\lambda}$ and the {\it Hankel operator} $H_f^{\lambda}$ defined on $\mathcal{A}_{\lambda}^2(\Omega)$ by
\begin{align*}
T_f^{\lambda}:
&=P_{\lambda} M_f,\\
H_f^{\lambda}:
&= (I-P_{\lambda}) M_f. 
\end{align*}
Here $M_f$ denotes the pointwise multiplication by $f$. A straightforward calculation shows the standard relation
\begin{equation*}
T_f^{\lambda}T_g^{\lambda} -T_{fg}^{\lambda} = -(H_{\overline{f}}^{\lambda})^*H_g^{\lambda},
\end{equation*}
which implies the norm estimate 
\begin{equation}\label{Semi_commutator_estimate}
\|T_f^{\lambda}T_g^{\lambda} -T_{fg}^{\lambda}\|_{\lambda}\leq \| H_{\overline{f}}^{\lambda}\|_{\lambda} \| H_g^{\lambda}\|_{\lambda}. 
\end{equation} 
\par  We are also concerned with Toeplitz operators having symbols in the space $\textup{UC}(\Omega)$ of complex valued functions on $\Omega$ that are uniformly continuous 
with respect to the Bergman metric distance $\beta$. Since $\textup{UC}(\Omega)$ contains unbounded functions (e.g.~$f(z):= \beta(0,z)$) Toeplitz operators with uniformly continuous 
symbols are unbounded in general (cf. Remark \ref{remark_main_result_1} and \cite{BC0}). Hence we need to define finite products of such operators in a careful way by specifying a common invariant dense domain. 
\vspace{1mm}\par 

We recall the Forelli-Rudin estimates \cite[Proposition 8]{E3}: let the BSD $\Omega \subset \mathbb{C}^n$ be of {\it type} $(r,a,b)$ with {\it characteristic multiplicities} $a,b \in \mathbb{Z}_+$. Then we have: 

\begin{lemma}\label{Forelli_Rudin_estimates_standard}
Let $\alpha > p-1$ and $t > \frac{r-1}{2}a$, then there is a constant $C > 0$ (independent of $z \in \Omega$) such that for all $z \in \Omega$ 
\begin{equation*}
\int_{\Omega} h(w,w)^{\alpha-p}|h(z,w)|^{-(\alpha+t)} \, dv(w) \leq Ch(z,z)^{-t}. 
\end{equation*}
\end{lemma}

Let $\rho > 0$ and consider the following function spaces: 
\begin{equation*}
S_{\rho}(\Omega):= \Big{\{} f \in C(\Omega) \: : \: \mbox{\it  $\exists \: C > 0$ s.t.~$|f(z)| \leq Ch(z,z)^{-\rho}$ for all $z \in \Omega$} \Big{\}}. 
\end{equation*}
Moreover, define $\mathcal{H}_{\rho}(\Omega): = S_{\rho}(\Omega) \cap \mathcal{A}^2_{\lambda}(\Omega)$ and consider the intersections
\begin{align*}
\textup{Sym}(\Omega) := \bigcap_{\rho >0} S_{\rho}(\Omega)
\hspace{3ex} \mbox{\it and} \hspace{3ex} 
\mathcal{D}:=\bigcap_{\rho > \rho^*} \mathcal{H}_{\rho}(\Omega), 
\end{align*}
where $\rho^* := \frac{r-1}{2}a$. Note that $\textup{Sym}(\Omega)$ is actually an algebra. Since $\mathcal{D}$ contains the restrictions of all holomorphic polynomials to $\Omega$ it is a dense subspace in all Bergman spaces $\mathcal{A}_{\lambda}^2(\Omega)$.

\begin{lemma}\label{Lemma_Abbildungseigenschaften_TO}
Let $\lambda > \rho^*+p-1$ and assume that $f \in \textup{Sym}(\Omega)$. Then the (possibly unbounded) Toeplitz operator $T^{\lambda}_f$ leaves the space 
$\mathcal{D}\subset \mathcal{A}_{\lambda}^2(\Omega)$ invariant. In particular, all finite products 
\begin{equation*}
T^{\lambda}_{f_1} T^{\lambda}_{f_2} \cdots T^{\lambda}_{f_m}: \mathcal{D} \longrightarrow \mathcal{D}
\end{equation*}
with symbols $f_j \in \textup{Sym}(\Omega)$ are defined and induce densely defined operators on $\mathcal{A}_{\lambda}^2(\Omega)$. 
\end{lemma}

\begin{proof}
Let $\rho \in (\rho^*, \lambda+1-p)$ and let $\epsilon > 0$ be sufficiently small such that $\alpha := -\rho-\epsilon+\lambda > p-1$. Given $g \in \mathcal{D}$ and  $f \in \textup{Sym}(\Omega)$ we can 
choose $C_{\rho}, c_{\epsilon} >0$ such that 
\begin{equation*}
|f(z)| \leq c_{\epsilon}h(z,z)^{-\epsilon} \hspace{2ex} \mbox{\it and } \hspace{2ex} |g(z)| \leq C_{\rho}h(z,z)^{-\rho}.
\end{equation*}
The Forelli-Rudin estimates in Lemma \ref{Forelli_Rudin_estimates_standard} imply then: 
\begin{align*}
\big{|}\big{[}T^{\lambda}_fg\big{]}(z) \big{|} 
&= \Big{|} \int_{\Omega} f(w)g(w)h(z,w)^{-\lambda} \, dv_{\lambda}(w) \Big{|}\\
&\leq C_{\rho} c_{\epsilon} \int_{\Omega} h(w,w)^{-\rho-\epsilon+\lambda-p}|h(z,w)|^{-\lambda} \, dv(w)\\
&\leq C_{\rho} c_{\epsilon} \int_{\Omega} h(w,w)^{\alpha-p}|h(z,w)|^{-(\alpha+\rho+\epsilon)} \, dv(w)\\
&\leq \widetilde{C}C_{\rho}c_{\epsilon}h(z,z)^{-(\rho+\epsilon)}. 
\end{align*}
Since $\rho > \rho^*$ and $\epsilon > 0$ can be chosen arbitrarily small, we conclude that $T_fg \in \mathcal{D}$. 
\end{proof}
Let $f \in \textup{BO}^{\lambda}(\Omega)$, then it follows from \eqref{Global_estimate_BO} with $z = 0$ that 
\begin{equation*}
|f(w)| \leq |f(0)| +|f(0)-f(w)| \leq |f(0)| + \|f \|_{\textup{BO}^{\lambda}} \big{(} 1+ \beta_{\lambda}(0,w)\big{)}.  
\end{equation*}
Corollary \ref{Corollary_Schur_test_hilf_1} below implies that for any $\rho > 0$ there is $C(\rho,f) > 0$ such that 
\begin{equation*}
|f(w)| \leq C(\rho, f)h(w,w)^{-\rho}, \hspace{4ex} w \in \Omega,
\end{equation*} 
and therefore one obtains the inclusions
\begin{equation}\label{Inclusion_UC_BO}
\textup{UC}(\Omega) \subset \textup{BO}^{\lambda}(\Omega) \subset \textup{Sym}(\Omega).
\end{equation}
\par 
Let $\mathcal{A}_{\textup{uc}}(\Omega)$ denote the (non-closed) subalgebra in $\textup{Sym}(\Omega)$ which is generated by functions in $\textup{UC}(\Omega)$, i.e.~$\mathcal{A}_{\textup{uc}}(\Omega)$ 
consists of finite sums of finite products of functions in $\textup{UC}(\Omega)$.  Then we have:
\begin{lemma}
Toeplitz operators with symbols $f \in \mathcal{A}_{\textup{uc}}(\Omega)$ leave $\mathcal{D}$ invariant. In particular, finite products of such operators with dense domain $\mathcal{D}$ are well defined.
\end{lemma}
%%%%%%%%%%%%%%%%%%%%%%%%%%%%%%%%%%%%%%%%%%%%%%%%%%%%%%%%%%%%%%%%%%%%%%%%%%%%%%%%%%%%%%%%%%%%%%%%%%%%%%%%
\section{Uniformly continuous functions and quantization}
\label{Section_uniformly_continuous}
\setcounter{equation}{0}
%%%%%%%%%%%%%%%%%%%%%%%%%%%%%%%%%%%%%%%%%%%%%%%%%%%%%%%%%%%%%%%%%%%%%%%%%%%%%%%%%%%%%%%%%%%%%%%%%%%%%%%
In the present section we study the asymptotic behavior of semi-commutators of Toeplitz operators with symbols in $f \in \textup{UC}(\Omega)$ (cf.~Theorem \ref{Theorem_behaviour_semi_commutator_uniformly_bounded_function}). Although each single 
Toeplitz operator $T_f^{\lambda}$ may be unbounded it follows from the inclusions (\ref{Inclusion_UC_BO}) together with the results in \cite{{BBCZ}} that the semi-commutators $T_g^{\lambda}T_f^{\lambda}-T_{fg}^{\lambda}$, 
$f,g \in \textup{UC}(\Omega)$ are bounded operators. We start with some preparations (Lemma \ref{BSD_auxiliary_lemma}, \ref{BSD_auxiliary_lemma2} and {Corollary}  \ref{Corollary_Schur_test_hilf_1}), which 
give auxiliary inequalities that are essential in the proof of  Proposition \ref{estimate_norm_Hankel_operator} devoted to the norm estimate of Hankel operators. 

\begin{lemma}\label{BSD_auxiliary_lemma}
Let $s,C_1,C_2 > 0$, $U \subset \C^n$ and let $f,g \colon U \to \R_+$ satisfy $g(z) \geq 1$ and $f(z) \leq C_1g(z)$ for all $z \in U$. Further assume that there exists a set $V \subset U$ 
such that 
\begin{equation*}
\begin{cases}
f(z) \leq C_2\sqrt{\log g(z)} & \text{for all} \;  z \in V, \\
g(z) \geq 1 + s& \text{\it for all} \;   z \in U \setminus V. 
\end{cases}
\end{equation*}
%f(z) \leq C_2\sqrt{\log g(z)}$ for all $z \in V$ and $g(z) \geq 1 + s$ for all $z \in U \setminus V$. 
%\vspace{1mm} \par 
Then there exists a constant $C^{\prime}> 0$ such that $\sqrt{\lambda}f(z) \leq C^{\prime}g(z)^{\lambda}$ for all $z \in U$ and $\lambda \geq 1$.
\end{lemma}

\begin{proof}
Fix $z \in V$ and $\lambda \geq 1$. If $g(z) = 1$, then $f(z) = 0$ and thus obviously $\sqrt{\lambda}f(z) \leq C^{\prime}g(z)^{\lambda}$ for all $\lambda \geq 1$ and any $C' > 0$. 
So assume that $g(z) > 1$ and set $C^{\prime}_V := \sqrt{\log 2}C_2$. Then
\[C^{\prime}_Vg(z)^{\lambda} = C^{\prime}_V2^{\frac{\lambda\log g(z)}{\log 2}} \geq \sqrt{\frac{\log 2}{\log g(z)}}f(z)\sqrt{\frac{\lambda\log g(z)}{\log 2}} = \sqrt{\lambda}f(z)\]
since $2^y \geq \sqrt{y}$ for all $y \geq 0$.

For $z \in U \setminus V$ we choose $C^{\prime}_{V^c} := C_1\frac{1+s}{\sqrt{2\log(1+s)}}$ so that
\[C^{\prime}_{V^c}g(z)^{\lambda} \geq C_1\frac{(1+s)^{\lambda}}{\sqrt{2\log(1+s)}}g(z) \geq C_1\sqrt{\lambda}g(z) \geq \sqrt{\lambda}f(z),\]
where we used $\frac{(1+s)^y}{\sqrt{2\log(1+s)}} \geq \sqrt{y}$ for all $y \geq 0$ and $s > 0$. Choosing $C' := \max\{C^{\prime}_V,C^{\prime}_{V^c}\}$ finishes the proof.
\end{proof}

\begin{lemma}\label{BSD_auxiliary_lemma2}
Let $\Omega \subset \C^n$ be a BSD and $\rho > 0$. Then there is a neighborhood $V$ of $0$ and a constant $C(\rho) > 0$ such that
\[\beta(0,z) \leq C(\rho)\sqrt{\log h(z,z)^{-\rho}}\]
for all $z \in V$.
\end{lemma}

\begin{proof}
Since $\sqrt{\log h(z,z)^{-\rho}} = \sqrt{\rho}\sqrt{-\log h(z,z)}$, it clearly suffices to check the assertion for one particular $\rho$. We may thus assume that $\rho = \lambda > p-1$.

With $s > 0$ and each fixed $z \in \Omega$ consider the polynomial 
\begin{equation*}
P_z(s):=h(sz,sz)=\prod_{j=1}^r\big{(}1-s^2t_j^2\big{)}= 1- s^2\sum_{j=1}^rt_j^2 + O(s^4) \hspace{4ex} (\mbox{as} \: \: s \downarrow 0). 
\end{equation*}
Therefore the Taylor expansion in $z=0$ of $z \mapsto h(z,z)$ cannot have a linear term. Write 
\begin{equation}\label{Expansion_h}
h(z,z) = 1+ \sum_{|\alpha+\beta| >1} a_{\alpha \beta} z^{\alpha} \overline{z}^{\beta} 
\end{equation}
and insert this expansion into \eqref{Formula_weighted_Bergman_kernel}:
\begin{equation*}
\log K_{\lambda}(z,z) = \log h(z,z)^{-\lambda} = -\lambda\log h(z,z)= -\lambda\log\Big{(} 1+ \sum_{|\alpha+\beta|>1} a_{\alpha \beta} z^{\alpha} \overline{z}^{\beta} \Big{)}.  
\end{equation*}
\par 
One obtains the Bergman metric tensor
\begin{align*}
g_{ij}(z):&=\frac{\partial^2}{\partial z_i \partial \overline{z}_j} \log K_{\lambda}(z,z)\\
&=\frac{\lambda}{h(z,z)^2} \Big{(}\sum_{|\alpha +\beta|>1} a_{\alpha \beta} \alpha_iz^{\alpha-e_i}\overline{z}^{\beta}\Big{)}\Big{(}\sum_{|\alpha+\beta|>1}a_{\alpha \beta} \beta_jz^{\alpha} \overline{z}^{\beta-e_j}\Big{)} 
\\
& \hspace{35ex} - \frac{\lambda}{h(z,z)} \sum_{|\alpha+\beta|>1} a_{\alpha \beta} \alpha_i \beta_j z^{\alpha-e_i}\overline{z}^{\beta -e_j},  
\end{align*}
where $e_i=(0,\ldots, 1, \ldots ,0) =(\delta_{i\ell})_{\ell=1, \ldots, n}\in \mathbb{Z}_+^n$. For $z=0$ we have $(g_{ij}(0))_{ij}= -\lambda (a_{e_{i} e_{j}})_{ij}$. Since the metric tensor is positive definite, it follows that 
\begin{equation*}
-A:=(a_{e_i e_j})_{ij}< 0. 
\end{equation*}
Hence we can write the quadratic term in the expansion of (\ref{Expansion_h}) as follows: 
\begin{equation}\label{Expansion_h_2}
h(z,z) = 1 - \langle Az,z\rangle + \sum_{|\alpha+\beta| >2} a_{\alpha \beta} z^{\alpha} \overline{z}^{\beta} 
\hspace{2ex} \mbox{\it where} \hspace{2ex} A= \frac{1}{\lambda} \big{(} g_{ij}(0)\big{)}_{i,j} >0.
\end{equation}
\par 
Let $\mu > 0$ denote the minimal eigenvalue of $(g_{ij}(0))_{ij}$. Then we can choose a convex zero-neighborhood $V \subset \Omega$ such that 
\begin{equation}\label{BK_estimate_1}
\log h(z,z)^{-\lambda} = \log K_{\lambda}(z,z) \geq - \lambda \log \Big{(}1 - \frac{\mu}{2} |z|^2\Big{)} \geq \frac{\lambda \mu}{2} |z|^2, \hspace{4ex} z \in V. 
\end{equation}
\par 
Given $z \in V$ we can consider the straight path
\[\gamma: [0,1] \rightarrow V: \gamma(t) = tz.\]
Then we can estimate:  
\begin{equation}\label{BK_estimate_2}
\beta(0,z) \leq \ell(\gamma)= \int_0^1 \sqrt{\sum_{i,j=1}^n g_{ij}(tz) z_i \overline{z}_j }dt \leq \sup_{z \in V} \sqrt{\| (g_{ij}(z))_{ij}\|}\cdot |z| =: C_V |z|, 
\end{equation}
where $|z|^2:= |z_1|^2 + \cdots + |z_n|^2$. A comparison of (\ref{BK_estimate_1}) and (\ref{BK_estimate_2}) gives for all $z \in V$: 
\begin{equation*}
\beta(0,z) \leq C_V|z| \leq \sqrt{\frac{2}{\lambda \mu}} C_V \sqrt{\log h(z,z)^{-\lambda}} =: C(\lambda) \sqrt{\log h(z,z)^{-\lambda}}. 
\end{equation*}
This finishes the proof. 
\end{proof}

\begin{corollary}\label{Corollary_Schur_test_hilf_1}
Let $\Omega \subset \C^n$ be a BSD and $\rho > 0$. Then there is a constant $C > 0$ (depending only on $\rho$) such that
\[\sqrt{\lambda}\beta(0,z) \leq Ch(z,z)^{-\rho\lambda}\]
for all $z \in \Omega$ and $\lambda \geq 1$.
\end{corollary}

\begin{proof}
It holds $h(z,z)^{-\rho} \geq 1$ and from \cite[Equation ($**$) on p.~317]{BBCZ} one has for all $z \in \Omega$: 
$$\beta(0,z) \leq C_1h(z,z)^{-\rho}.$$
Furthermore, Lemma \ref{BSD_auxiliary_lemma2} implies $\beta(0,z) \leq C_2\sqrt{\log h(z,z)^{-\rho}}$ for all $z$ in a suitable zero-neighborhood $V \subset \Omega$.  
The product form (\ref{GL_expansion_on_the_diagonal}) shows that $h(z,z)^{-\rho} > 1$ for $z \neq 0$ and $h(z,z)^{-\rho} \to \infty$ as $z \to \partial\Omega$. Therefore one also has 
$$h(z,z)^{-\rho} \geq 1 + s$$ 
for some $s > 0$ and all $z \in \Omega \setminus V$. Setting $f(z) := \beta(0,z)$ and $g(z) := h(z,z)^{-\rho}$, the result follows from Lemma \ref{BSD_auxiliary_lemma}.
\end{proof}
\begin{proposition}\label{estimate_norm_Hankel_operator}
Let $f \in \textup{BO}^{\lambda}(\Omega)$. Then there is a constant $C > 0$, independent of $f$ and of $\lambda > 4p$, such that 
\begin{equation*}
\| H_f^{\lambda}\|_{\lambda}  \leq C\| f\|_{\textup{BO}^{\lambda}}.
\end{equation*}
\end{proposition}
\begin{proof}
Let $g \in \mathcal{A}_{\lambda}^2(\Omega)$. From the integral expression of the Hankel operator
\begin{equation*}
\big{[} H^{\lambda}_fg\big{]}(z)= \int_{\Omega} \big{[} f(z)-f(w) \big{]} g(w) K_{\lambda}(z,w) dv_{\lambda}(w), \hspace{3ex} z \in \Omega,
\end{equation*}
and the estimate (\ref{Global_estimate_BO}) in Lemma \ref{GL_estimate_global_BO} it follows that
\begin{equation*}
\big{|}H_f^{\lambda}g(z)\big{|}\leq \|f\|_{\textup{BO}^{\lambda}} \int_{\Omega} (\beta_{\lambda}(z,w)+1)|h(z,w)|^{-\lambda} |g(w)| \, dv_{\lambda}(w). 
\end{equation*}
The constant $C>0$ can be chosen as the norm of the integral operator
\begin{equation*}
L^2(\Omega, dv_{\lambda}) \ni u \mapsto \mathcal{L}_{\lambda}(u)(z):= \int_{\Omega} \underbrace{(\beta_{\lambda}(z,w)+1)|h(z,w)|^{-\lambda}}_{=:L_{\lambda}(z,w)} 
u(w) \, dv_{\lambda}(w) \in L^2(\Omega,dv_{\lambda}). 
\end{equation*}
\par 
In order to estimate the norm (independently of $\lambda$) we apply the Schur test. Since the integral kernel $L_{\lambda}(z,w)$ of $\mathcal{L}_{\lambda}$ is symmetric in $z$ and $w$ it 
is sufficient to construct a positive function $h$ on $\Omega$ and a constant $C > 0$ independent of $\lambda$ such that for all $z \in \Omega$:
\begin{equation*}
I_{\lambda}(z):=\int_{\Omega} (\beta_{\lambda}(z,w)+1)|h(z,w)|^{-\lambda} h(w) \, dv_{\lambda}(w) \leq C h(z). 
\end{equation*}
\par 
With $t$ to be determined later put
\[h(z):= h(z,z)^t,\]
and let $\varphi_z$ be an involutive automorphism of $\Omega$ interchanging $0$ and $z$. A change of variables and the identity $\beta_{\lambda}(z,w)= \beta_{\lambda}(0, \varphi_z(w))$ gives:
\begin{align*}
I_{\lambda}(z)
 &= c_{\lambda} \int_{\Omega} \big{[}\beta_{\lambda}(0,\varphi_z(w))+1\big{]} |h(z,w)|^{-\lambda}h(w,w)^{t+\lambda-p} \, dv(w) \\
 &= c_{\lambda} \int_{\Omega} \big{[}\beta_{\lambda}(0,w)+1\big{]} |h(z,\varphi_z(w))|^{-\lambda}h(\varphi_z(w),\varphi_z(w))^{t+\lambda-p} \, dv(\varphi_z(w)) =(+). 
\end{align*}
\par 
We can use the following standard relations (see e.g.~\cite{E3})
\begin{equation*}
h(z,\varphi_z(w)) = \frac{h(z,z)}{h(z,w)} \hspace{3ex} \mbox{\it and} \hspace{3ex} h(\varphi_z(w),\varphi_z(w)) = \frac{h(z,z)h(w,w)}{|h(z,w)|^2},
\end{equation*}
as well as
\begin{equation*}
dv(\varphi_z(w)) = \left(\frac{h(z,z)}{|h(z,w)|^2}\right)^p \, dv(w),
\end{equation*}
to obtain
\begin{equation*}
(+) = c_{\lambda}\underbrace{h(z,z)^t}_{=h(z)}\int_{\Omega} [\beta_{\lambda}(0,w)+1] h(w,w)^{t+\lambda-p}|h(z,w)|^{-2t-\lambda} \, dv(w).
\end{equation*}
\par 
Now we apply Corollary \ref{Corollary_Schur_test_hilf_1}. For any $\rho > 0$ there is a constant $C = C(\rho) > 0$ (independent of $\lambda$ and $w \in \Omega$) such that
\begin{equation}\label{Estimate_Bergman_metric_puls_one}
1 + \beta_{\lambda} (0,w) = 1 + \sqrt{\frac{\lambda}{p}}\beta(0,w) \leq C(\rho)h(w,w)^{-\rho\lambda}.
\end{equation}
Hence we can further estimate $(+)$ by
\begin{equation*}
(+) \leq c_{\lambda}C(\rho)h(z)\int_{\Omega} h(w,w)^{t+(1-\rho)\lambda-p}|h(z,w)|^{-2t-\lambda} \, dv(w) = (++). 
\end{equation*}
Choosing $t= -\frac{\lambda}{2}$ and $\rho= \frac{1}{4}$, we obtain
\[(++) = c_{\lambda}C\Big{(}\frac{1}{4}\Big{)}h(z) \int_{\Omega}h(w,w)^{\frac{\lambda}{4}-p} \, dv(w) = \frac{c_{\lambda}}{c_{\frac{\lambda}{4}}}C\Big{(}\frac{1}{4}\Big{)}h(z).\]
Since $\frac{c_{\lambda}}{c_{\frac{\lambda}{4}}}$ is bounded (as a function of $\lambda$, cf. \cite{FK}), the result follows from the Schur test.
\end{proof}
A relation between $\textup{BMO}^{\lambda}(\Omega)$ and $\textup{BO}^{\lambda}(\Omega)$ is given by the following result: 
\begin{theorem}\label{Theorem_Estimate_BMO_BO}
Let $g \in \textup{BMO}^{\lambda}(\Omega)$ and $\lambda \geq p$. Then we have for all $z,w \in \Omega$ 
\begin{equation*}
\big{|} \mathcal{B}_{\lambda}(g)(w)-\mathcal{B}_{\lambda}(g)(z) \big{|} \leq  2 \| g\|_{\textup{BMO}^{\lambda}} \beta_{\lambda} (z,w). 
\end{equation*}
In particular,  $\mathcal{B}_{\lambda}(g) \in \textup{BO}^{\lambda}(\Omega)$ and 
\begin{equation*}
\| \mathcal{B}_{\lambda}(g) \|_{\textup{BO}^{\lambda}} \leq 2 \| g\|_{\textup{BMO}^{\lambda}}. 
\end{equation*}
\end{theorem}
\begin{proof}
See \cite[Theorem 4.9]{BC1}. 
\end{proof}
Let $\lambda > 4p$, replace $f\in \textup{BO}^{\lambda}(\Omega)$ in Proposition \ref{estimate_norm_Hankel_operator}  by $\mathcal{B}_{\lambda}(f)$,  where $f \in \textup{BMO}^{\lambda}(\Omega)$, 
and use Theorem \ref{Theorem_Estimate_BMO_BO}. We obtain a constant $C>0$ independent of $f$ and $\lambda $ such that
\begin{equation}\label{inequality_operator_norm_weighted_Hankel_operator}
\| H_{\mathcal{B}_{\lambda}(f)}^{\lambda} \|_{\lambda} \leq C  \|\mathcal{B}_{\lambda}(f) \|_{\textup{BO}^{\lambda}} \leq 2C \| f\|_{\textup{BMO}^{\lambda}}. 
\end{equation}
\par 
In particular, let $f \in \textup{UC}(\Omega)$ be uniformly continuous w.r.t.~the Bergman metric $\beta(z,w)$. We analyze the asymptotic behavior of $\| f\|_{\textup{BMO}^{\lambda}}$ as $\lambda \rightarrow \infty$. 
By applying a  change of variables in the integral, the Berezin transform (\ref{Defintion_Berezin_Transform_of_a_function}) of a function $f$ can be represented as a convolution type integral (cf. \cite{BC1}): 
\begin{equation}\label{Berezin_transform_second_form}
\mathcal{B}_{\lambda}(f)(x)=c_{\lambda}\int_{\Omega} (f \circ \varphi_x)(y) h(y,y)^{\lambda-p} \, dv(y). 
\end{equation}
\par 
We will use the following asymptotic behavior of the Berezin transform for uniformly continuous symbols: 
\begin{proposition}\label{Lemma_convergence_Berezin_transform_UC_functions}
Let $\Omega\subset \mathbb{C}^n$ be a BSD and $f \in \textup{UC}(\Omega)$. Then
\[\lim\limits_{\lambda \to \infty} \mathcal{B}_{\lambda}(f) = f,\]
where the convergence is uniformly on $\Omega$. 
\end{proposition}

\begin{proof}
See \cite[Proposition 4.4]{BC1}. 
\end{proof}
According to (\ref{Definition_rem_mean_oscillation}) and (\ref{Berezin_transform_second_form}), we can write
\begin{align*}
\textup{MO}^{\lambda}(f)(x)
&= c_{\lambda} \int_{\Omega} \big{|} f\circ \varphi_x(y)-\mathcal{B}_{\lambda}(f)(x)|^2 h(y,y)^{\lambda-p} \, dv(y) \\
&= c_{\lambda} \int_{\Omega} \big{|} f \circ \varphi_x(y) - \mathcal{B}_{\lambda} (f) \circ \varphi_x(0) \big{|}^2 h(y,y)^{\lambda-p} \, dv(y). 
\end{align*}
The next observation is crucial in the proof of our main theorem:
\begin{proposition}\label{Lemma_limit_BMO_lambda_UC}
Let $f \in \textup{UC}(\Omega)$. Then $\lim_{\lambda \rightarrow \infty} \| f\|_{\textup{BMO}^{\lambda}} = 0$. 
\end{proposition}
\begin{proof}
Let $\frac{1}{4} > \epsilon > 0$ be fixed and choose $\delta >0$ such that $|f(z)-f(w)| < \varepsilon$ for all $z,w \in \Omega$ with $\beta(z,w) < \delta$. We divide the domain of integration into two parts: 
\begin{equation}\label{decomposition_MO_lambda}
\textup{MO}^{\lambda}(f)(x)
= c_{\lambda} \Big{\{}\int_{\beta(y,0) < \delta} +\int_{\beta(y,0) \geq \delta} \Big{\}}\big{|} f \circ \varphi_x(y) - \mathcal{B}_{\lambda}(f) \circ \varphi_x(0) \big{|}^2 h(y,y)^{\lambda-p} \, dv(y).
\end{equation}
In the case $\beta(y,0) < \delta$ we have
\begin{equation*}
\beta(\varphi_x(y),\varphi_x(0)) = \beta(y,0) < \delta 
\end{equation*}
uniformly for all $x \in \Omega$.  The uniform continuity of $f$ and Proposition \ref{Lemma_convergence_Berezin_transform_UC_functions} for sufficiently large weight parameter $\lambda$ imply then that
\begin{multline*}
| f \circ \varphi_x(y) - \mathcal{B}_{\lambda}(f) \circ \varphi_x(0)|^2\leq \\
\leq \Big{(} \big{|} f \circ \varphi_x(y)-f \circ \varphi_x(0)\big{|} + \big{|} (f- \mathcal{B}_{\lambda}(f)) \circ \varphi_x(0)\big{|}\Big{)}^2 < 4 \varepsilon^2 < \varepsilon. 
\end{multline*}
Hence  we obtain
\begin{equation*}
0 \leq \textup{MO}^{\lambda}(f)(x)\leq \varepsilon + c_{\lambda} \int_{\beta(y,0) \geq \delta}\big{|} f \circ \varphi_x(y) - \mathcal{B}_{\lambda}(f) \circ \varphi_x(0) \big{|}^2 h(y,y)^{\lambda-p} \, dv(y)=(+).
\end{equation*}
It is known (see \cite[Lemma 2.1]{BC1}  or (\ref{Inclusion_UC_BO})) that $\textup{UC}(\Omega) \subset \textup{BO}(\Omega)$ and therefore 
\begin{align}\label{MO_first_estimate}
\big{|} f \circ \varphi_x(y)- f \circ \varphi_x(0) \big{|} 
&\leq\|f\|_{\textup{BO}^p} \Big{[} 1+ \beta_p\big{(}\varphi_x(y), \varphi_x(0)\big{)}\Big{]}\\
& = \| f \|_{\textup{BO}^p}\Big{[} 1+ \beta(y,0)\Big{]}. \notag
\end{align}
\par The difference $f- \mathcal{B}_{\lambda}(f)$ is uniformly bounded on $\Omega$. According to Proposition \ref{Lemma_convergence_Berezin_transform_UC_functions}, we have  
$$\lim_{\lambda \rightarrow \infty} \|f- \mathcal{B}_{\lambda}(f) \|_{\infty}=0$$ 
and therefore it follows for all $x \in \Omega$ and sufficiently large parameter $\lambda$ that 
\begin{equation*}
| f \circ \varphi_x(y) - \mathcal{B}_{\lambda}(f) \circ \varphi_x(0)|\leq 2 \| f\|_{\textup{BO}^p} \big{[}1+ \beta(y,0) \big{]}. 
\end{equation*}
\par 
The last estimate implies that for sufficiently large  $\lambda$ and all $x \in \Omega$ we have 
\begin{equation*}
(+)\leq \varepsilon + 4c_{\lambda}\| f \|^2_{\textup{BO}^p} \int_{\beta(y,0) \geq \delta}\Big{[} 1+ \beta(y,0)\Big{]}^2 h(y,y)^{\lambda-p} \, dv(y). 
\end{equation*}
\par 
Since $g(y):= 1+ \beta(y,0)$ defines an element in $L^2(\Omega, dv)$ (see \cite[Theorem E]{BBCZ}), $c_{\lambda} \sim \lambda^{n}$ as $\lambda \rightarrow \infty$ 
(see e.g.~\cite{BC1}) and there is a constant $s > 0$ (only depending on $\delta$) such that $h(z,z) \leq 1-s$ for all $z \in \Omega$ with $\beta(z,0) \geq \delta$, it follows that
\begin{equation*}
\lim_{\lambda \rightarrow \infty} 4 c_{\lambda}\| f \|^2_{\textup{BO}^p} \int_{\beta(y,0) \geq \delta}\Big{[} 1+ \beta(y,0)\Big{]}^2 h(y,y)^{\lambda-p} \, dv(y) = 0. 
\end{equation*}
This implies that $\lim\limits_{\lambda \rightarrow \infty} \|f\|_{\textup{BMO}^{\lambda}} = 0$.
\end{proof}
By combining the previous estimates, we obtain the following {\it quantization result} on the semi-commutator of Toeplitz operators with symbols in $\textup{UC}(\Omega)$. 
\begin{theorem}\label{Theorem_behaviour_semi_commutator_uniformly_bounded_function}
Let $f \in \textup{UC}(\Omega)$. Then $\lim\limits_{\lambda \to \infty} \|H_f^{\lambda}\|_{\lambda} = 0$. 
In particular,  
\begin{equation}\label{limit_semi_commutator}
\lim_{\lambda \rightarrow \infty} \big{\|}T_f^{\lambda} T_g^{\lambda}- T_{fg}^{\lambda} \big{\|}_{\lambda}=0 
\end{equation}
for all $g \in L^{\infty}(\Omega)$ or all $g \in \textup{UC}(\Omega)$. 
\end{theorem}
\begin{proof}
According to (\ref{Semi_commutator_estimate}), it is sufficient to show that $\lim_{\lambda \rightarrow \infty} \| H_f^{\lambda}\|_{\lambda}=0$. We use the estimate: 
\begin{align*}
\|H_f^{\lambda} \|_{\lambda} 
& \leq \big{\|} H_{f-\mathcal{B}_{\lambda}(f)}^{\lambda}\big{\|}_{\lambda} + \| H_{\mathcal{B}_{\lambda}(f)}^{\lambda} \big{\|}_{\lambda} 
 \leq \big{\|} f- \mathcal{B}_{\lambda}(f)\big{\|}_{\infty} + \| H_{\mathcal{B}_{\lambda}(f)}^{\lambda} \big{\|}_{\lambda}. 
\end{align*}
\par 
From Proposition \ref{Lemma_convergence_Berezin_transform_UC_functions} we conclude that the first summand on the right-hand side tends to zero as $\lambda \to \infty$. Proposition 
 \ref{Lemma_limit_BMO_lambda_UC} together with estimate (\ref{inequality_operator_norm_weighted_Hankel_operator})
 implies that $\lim\limits_{\lambda \to \infty}  \|H_{\mathcal{B}_{\lambda}(f)}^{\lambda} \big{\|}_{\lambda} =~0$, and the assertion follows. 
\end{proof}
%}
\begin{remark}\label{remark_main_result_1}
\textup{
As was previously mentioned, the space $\textup{UC}(\Omega)$ contains unbounded functions.  In \cite[Theorem 3.8]{BC0} the following equivalence is shown for Toeplitz operators 
with symbols in $\textup{UC}(\Omega)$:  
\begin{equation*}
\mbox{\it $T_f$ is bounded, $f \in \textup{UC}(\Omega)$} \hspace{3ex} \Longleftrightarrow \hspace{3ex} f \in \textup{UC}(\Omega) \cap L^{\infty}(\Omega).
\end{equation*}}
\textup{ 
According to (\ref{Inclusion_UC_BO}), we have the inclusions
\begin{equation*}
\textup{UC}(\Omega) \subset \textup{BO}^{\lambda}(\Omega) \subset \textup{BMO}^{\lambda}(\Omega). 
\end{equation*} 
Moreover, $H_f^{\lambda}$ is bounded in case of $f\in \textup{BMO}^{\lambda}(\Omega)$. In particular, the Hankel operator $H_f^{\lambda}$ with uniformly continuous symbol $f$ is bounded. 
Therefore the semi-commutators in (\ref{limit_semi_commutator}) are bounded although each single Toeplitz operator may be unbounded (cf.~Section \ref{UC_products}).}
\vspace{1ex}\par 
\textup{
If one prefers to deal with bounded Toeplitz operators, one may choose the symbols from the space ($C^*$-algebra) $ \textup{BUC}(\Omega)$ of bounded $\beta$-uniformly continuous functions on $\Omega$. Note that 
\begin{equation*}
C(\overline{\Omega}) \subsetneq \textup{BUC}(\Omega). 
\end{equation*}
In this case a stronger version of Theorem \ref{Theorem_behaviour_semi_commutator_uniformly_bounded_function} holds true:
\begin{corollary}
Let $f_1, \cdots f_m \in \textup{BUC}(\Omega)$. Then we have 
\begin{equation*}
\lim_{\lambda \rightarrow \infty}\|T_{f_1}^{\lambda} T_{f_2}^{\lambda} \cdots T_{f_m}^{\lambda} - T_{f_1f_2 \cdots f_m}^{\lambda}\|_{\lambda} =0. 
\end{equation*}
\end{corollary}
\begin{proof}
Use Theorem \ref{Theorem_behaviour_semi_commutator_uniformly_bounded_function} and standard estimates.
\end{proof}
}
\end{remark}

We draw some further conclusions and comment on the compactness of semi-commutators. 
\begin{lemma} \label{le:intersection}
Let $\Omega$ be a BSD  in $\mathbb{C}^n$, then 
$\textup{VMO}_{\partial}(\Omega) \cap \textup{UC}(\Omega) = \textup{VO}_{\partial}(\Omega).$
\end{lemma}
\begin{proof}
Let $f \in \textup{VO}_{\partial}(\Omega)$. Given $\varepsilon >0$, select a compact subset $K$ of $\Omega$ such that for each $z \in \Omega \setminus K$ we have that 
$$\text{Osc}^p_z(f) = \sup\big{\{}|f(z) - f(w)|\,: \: w \in \Omega, \;  \beta(z,w) < 1\big{\}} < \varepsilon.$$ 
The restriction $f|_K$ is obviously uniformly continuous. Thus we can find $\delta$, which depends on the given $\varepsilon$ so that $|f(z) - f(w)| < \varepsilon$ whenever $\beta(z,w) < \delta$.
That is $\textup{VO}_{\partial}(\Omega) \subset \textup{UC}(\Omega)$.

Furthermore, by \cite[Theorem B]{BBCZ}, 
\[\textup{VMO}_{\partial}(\Omega) = \textup{VO}_{\partial}(\Omega) + \mathcal{J},\]
where, by \cite[pages 940 and 944]{BCZ}, $\mathcal{J}$ consists of all functions $g \in \textup{VMO}_{\partial}(\Omega)$ such that the Toeplitz operator $T_g$ is compact. 
Now, $\textup{VO}_{\partial}(\Omega) \subset \textup{UC}(\Omega)$ implies that
\begin{equation*}
\textup{VMO}_{\partial}(\Omega) \cap \textup{UC}(\Omega) = \textup{VO}_{\partial}(\Omega) + \mathcal{J} \cap \textup{UC}(\Omega).
\end{equation*}
By \cite[Theorem 3.8]{BC0}, $\mathcal{J} \cap \textup{UC}(\Omega) = C_0(\Omega) \subset \textup{VO}_{\partial}(\Omega)$, which finishes the proof.
\end{proof}

Consider the function space
$$\Gamma := \big{\{} f \in L^{\infty}(\Omega)\, : \ T_g^{\lambda}T_f^{\lambda} -T_{gf}^{\lambda} \ \ \text{\it is compact} \ \text{\it for all} \ g \in L^{\infty}(\Omega) \big{\}}.$$

We summarize the results of \cite[Theorem B, Section 9]{BBCZ}, \cite[Theorem A, Proposition 1]{BCZ}, and \cite[Proposition 6]{Zhu} in the following statement. 
Note that, although the paper \cite{Zhu} is devoted to the case of the unit disk, the result of its Proposition 6 remains valid for the case of a general bounded symmetric domain $\Omega$. 
\begin{theorem}
 Let $f \in L^{\infty}(\Omega)$. Then the following statements are equivalent:
 \begin{enumerate}
  \item $f \in \Gamma \cap \overline{\Gamma}$,
  \item $f \in \textup{VO}_{\partial}(\Omega) + \mathcal{J}$,
  \item $H_f$ and $H_{\overline{f}}$ are compact,
  \item $[P_{\lambda}, M_f]$ is compact,
  \item $T_f^{\lambda}T_{\overline{f}}^{\lambda} -T_{|f|^2}^{\lambda}$ and $T_{\overline{f}}^{\lambda}T_f^{\lambda} -T_{|f|^2}^{\lambda}$ are compact.
 \end{enumerate}
\end{theorem}
 
Thus the operator-theoretic version of Lemma \ref{le:intersection} reads as follows.
\begin{corollary}
 Let $f \in \textup{BUC}(\Omega)$. Then \textup{(i)} and \textup{(ii)} are equivalent: 
 \begin{itemize}
 \item[\textup{(i)}]  Both semi-commutators $T_f^{\lambda}T_g^{\lambda} -T_{fg}^{\lambda}$ and $T_g^{\lambda}T_f^{\lambda} -T_{gf}^{\lambda}$ 
 are compact for all $g \in L^{\infty}(\Omega)$, 
 \item[\textup{(ii)}]  $f \in \textup{VO}_{\partial}(\Omega)$.
 \end{itemize}
\end{corollary}
We mention (cf. \cite[page 924]{BCZ}) that, in the case of $\Omega= \mathbb{B}^n$, the algebra $C(\overline{\mathbb{B}^n})$ is a subset of $\textup{VO}_{\partial}(\mathbb{B}^n)$, and that this inclusion fails for higher rank domains. That is, in the classical situation of $\Omega= \mathbb{B}^n$ and $f \in C(\overline{\mathbb{B}^n})$, the  semi-commutator 
\begin{equation*}
T_f^{\lambda} T_g^{\lambda}-T_{fg}^{\lambda}
\end{equation*}
of Theorem \ref{Theorem_behaviour_semi_commutator_uniformly_bounded_function} is compact for each $g \in L^{\infty}(\mathbb{B}^n)$ and all $\lambda >p-1=n$.
However, in case of operator symbols from $\textup{BUC}(\Omega)$ such compactness does not need to be true.

\begin{example}
Let $\Omega = \mathbb{D}$ be the unit disk. Given a point $t_0$ on the unit circle $S^1 =\partial \mathbb{D}$, let $\ell_{t_0}$ be the arc on $S^1$ with endpoint $t_0$ and $-t_0$. Define then the function 
\begin{equation*}
 f_{t_0}(z) = 2 \omega(z,\ell_{t_0},\mathbb{D}) -1,\qquad z \in \mathbb{D},
\end{equation*}
where 
\begin{equation*}
 \omega(z,\ell_{t_0},\mathbb{D}) = \int_{\ell_{t_0}} \frac{1 - |z|^2}{|e^{i\theta} - z|^2} \frac{d\theta}{2\pi}
\end{equation*}
is the harmonic measure of $\ell_{t_0}$ at $z$ in $\mathbb{D}$, cf. \cite[page 5]{GarMar}. The function $f_{\ell_{t_0}}$, being bounded and harmonic, belongs to $\textup{BUC}(\mathbb{D})$.
At the same time $f_{\ell_{t_0}}$ is a \emph{boundary piecewise continuous} function in the sense of \cite[Section 5]{Va07}. 

Theorem 5.1 of \cite{Va07} describes the quotient algebra of the algebra generated by Toeplitz operators with boundary piecewise continuous symbols modulo compact operators. This description 
implies that the self semi-commutator $T_{f_{t_0}}^{\lambda} T_{f_{t_0}}^{\lambda}-T_{f_{t_0}^2}^{\lambda}$ is not compact, while for any point  $t_1 \in S^1$ different from $\pm t_0$, the semi-commutator 
$T_{f_{t_0}}^{\lambda} T_{f_{t_1}}^{\lambda}-T_{f_{t_0}f_{t_1}}^{\lambda}$ is compact. 

Note that the paper \cite{Va07} deals with the classical Bergman space {\rm(}$\lambda = p${\rm )}, but the results therein remain valid (with correspondent adjustments in formulas) for each weighted 
Bergman space $\mathcal{A}^2_{\lambda}(\mathbb{D})$ with $\lambda > p-1$.
\end{example}
\par 
The following example shows that Theorem \ref{Theorem_behaviour_semi_commutator_uniformly_bounded_function} may fail for bounded symbols  being not continuous in $\Omega$ 
(cf.~\cite[Example 5.1]{BC-1}). 

\begin{example}\label{counterexample_fast_oscillating_symbols}
We consider the unit disc $\Omega = \mathbb{D}$ in $\mathbb{C}$. Let $f \in L^{\infty}(\mathbb{D})$ be defined by
\[f(z) := \begin{cases} 1 & \text{if } z = 0, \\ e^{i|z|^{-2}} & \text{if } z \neq 0. \end{cases}\]
Since $\bar{f}f = 1$, we get $T_{\bar{f}f}^{\lambda} = \textup{Id}$ for all $\lambda = 2+\alpha$, where $\alpha > -1$ (see below). We will now show that $T_f^{\lambda}1 \to 0$ as $\lambda \to \infty$, 
which contradicts the statement of Theorem \ref{Theorem_behaviour_semi_commutator_uniformly_bounded_function}. Since $f$ is radially symmetric, \cite[Theorem 3.1]{GKV} implies
\begin{equation}\label{first_integral}
(T_f^{\lambda}1)(z) = (\alpha+1)\int_0^1 e^{ir^{-1}} (1-r)^{\alpha} \, dr = \frac{\alpha+1}{\alpha}\int_{\frac{1}{\alpha}}^\infty e^{i\alpha x} \left(1-\frac{1}{\alpha x}\right)^{\alpha} \frac{1}{x^2} \, dx =(+),
\end{equation}
where we have used the substitution $x := \frac{1}{\alpha}r^{-1}$ (for $\alpha \geq 1$, say).
We change variables again and put $y := x + \frac{\pi}{\alpha}$. This yields
\begin{equation} \label{second_integral}
(+)=-\frac{\alpha+1}{\alpha}\int_{\frac{1+\pi}{\alpha}}^\infty e^{i\alpha y} \left(1-\frac{1}{\alpha y - \pi}\right)^{\alpha} \frac{1}{(y-\frac{\pi}{\alpha})^2} \, dy.
\end{equation}
By taking the average of \eqref{first_integral} and \eqref{second_integral}, we get
\begin{align*}
(T_f^{\lambda}1)(z) = \frac{\alpha+1}{2\alpha}\int_0^\infty e^{i\alpha x}\left[\left(1-\frac{1}{\alpha x}\right)^{\alpha} \frac{\chi_{[\frac{1}{\alpha},\infty)}(x)}{x^2} - \left(1-\frac{1}{\alpha x - \pi}\right)^{\alpha} \frac{\chi_{[\frac{1+\pi}{\alpha},\infty)}(x) }{(x-\frac{\pi}{\alpha})^2}\right] \, dx.
\end{align*}
The integrand on the right-hand side is uniformly bounded by
\[e^{-\frac{1}{x}}\frac{1}{x^2} + e^{-\frac{1}{x}}\min\left\{1,\frac{1}{(x-\pi)^2}\right\} \in L^1([0,\infty))\]
for $\alpha = \lambda-2 \geq 1$ and converges pointwise to $0$. Thus  $T_f^{\lambda}1 \to 0$ as $\lambda \rightarrow \infty$ by the dominated convergence theorem. 
\end{example}
%%%%%%%%%%%%%%%%%%%%%%%%%%%%%%%%%%%%%%%%%%%%%%%%%%%%%%%%%%%%%%%%%%%%%%%%%%%%%%%%%%%%%%%%%%%%%%%%%%
\section{Quantization and "VMO inside"}
\label{Section_VMO}
\setcounter{equation}{0}
%%%%%%%%%%%%%%%%%%%%%%%%%%%%%%%%%%%%%%%%%%%%%%%%%%%%%%%%%%%%%%%%%%%%%%%%%%%%%%%%%%%%%%%%%%%%%%%%%%%%
Let $\Omega=\mathbb{B}^n \subset \mathbb{C}^n$ denote the open Euclidean unit ball. In this case the genus is $p=n+1$ and the rank $r=1$. As usual, we put $\lambda = n+1+\alpha$, where $\alpha > -1$. With our previous notation and $\lambda > n=p-1$ the weighted measure on $\mathbb{B}^n$ is given by
\begin{equation*}
dv_{\lambda}(y) = c_{\lambda}h(y,y)^{\lambda-p} \, dv(y)= \frac{\Gamma(n+1+\alpha)}{n!\Gamma(\alpha+1)}(1-|y|^2)^{\alpha} dv(y).
\end{equation*}
\par 
With $\rho>0$ and $x \in \mathbb{B}^n$ consider the Bergman balls 
\begin{equation*}
E(x,\rho):= \big{\{} y \in \mathbb{B}^n \: : \: \beta(x,y) < \rho \big{\}} \hspace{2ex} \mbox{\it with volume} \hspace{2ex} |E(x,\rho)|= \int_{E(x,\rho)} 1 \, dv(y). 
\end{equation*}
For a locally integrable function $f$ on $\mathbb{B}^n$ and by using the notation in \cite{BCZ} we define the averaging function
\begin{equation*}
\hat{f}(x,\rho):= \frac{1}{|E(x,\rho)|} \int_{E(x,\rho)} f(y) dv(y). 
\end{equation*}
With $q \geq 1$ and $f \in L^q (\mathbb{B}^n)$ put now
\begin{equation*}
A_q(f,\rho,x):= \frac{1}{|E(x,\rho)|} \int_{E(x,\rho)} |f(y)-\hat{f}(x,\rho)|^q dv(y). 
\end{equation*}

\begin{definition}
With $q=2$ we define the space of bounded functions that have vanishing oscillation inside the unit ball (cf. \cite{Zhu}) 
\begin{equation}\label{defn_VMO_inside}
\textup{VMO}(\mathbb{B}^n):= \Big{\{} f \in L^{\infty}(\mathbb{B}^n) \: : \:  \lim_{\rho \rightarrow 0} A_2(f,\rho,x) =0 \; \textup{\it  uniformly for } \: x \in \mathbb{B}^n\Big{\}}. 
\end{equation}
Note that different from standard notations we assume functions in $\textup{VMO}(\mathbb{B}^n)$ to be bounded. 
\end{definition}
\begin{remark}
A standard estimate shows that 
\begin{equation*}
A_2(f,\rho,x) \leq \frac{1}{|E(x,\rho)|^2} \int_{E(x, \rho) \times E(x,\rho)} \big{|} f(y)-f(z)\big{|}^2dv(y,z). 
\end{equation*}
\end{remark}
Note that one has the proper inclusion $\textup{BUC}(\mathbb{B}^n) \subsetneq \textup{VMO}(\mathbb{B}^n)$. Here is an example of a function in $\textup{VMO}(\mathbb{D})$ that is not continuous:

\begin{example}
Let $\Omega= \mathbb{D}= \mathbb{B}^1$ and with $r \in (0,1)$ put 
$$f(r) := \log\Big{(}\log\Big{(}1 + \frac{1}{r}\Big{)}\Big{)}.$$
Consider $g = f \circ |\cdot| : \mathbb{D} \setminus \{0\} \to \mathbb{C}$. This function is clearly not continuous at $0$. Moreover, $f$ is convex and
\[|f'(r)| = \left|\frac{1}{\log(1 + \frac{1}{r})} \frac{1}{1 + \frac{1}{r}}\frac{1}{r^2}\right| \leq \frac{1}{r\log(1 + \frac{1}{r})}.\]
Thus
\begin{align*}
A_2(g,\rho,x) &\leq \frac{1}{|E(x,\rho)|^2} \int_{E(x,\rho)^2} \big{|} f(|y|)-f(|z|)|^2 \, dv(y,z)\\
&= \frac{2}{|E(x,\rho)|^2} \int\limits_{\substack{E(x,\rho)^2 \\ |y| \leq |z|}} \big{|}f(|y|)-f(|z|)\big{|}^2 \, dv(y,z)\\
&\leq \frac{2}{|E(x,\rho)|^2} \int\limits_{\substack{E(x,\rho)^2 \\ |y| \leq |z|}} \big{|}f'(|y|)|^2 \underbrace{(|y|-|z|)^2}_{\leq \frac{4}{\pi} |E(x,\rho)|} \, dv(y,z)\\
&\leq \frac{8}{\pi} \int_{E(x,\rho)} \frac{1}{(|y|\log(1 + \frac{1}{|y|}))^2} \, dv(y).
\end{align*}
Since $\frac{1}{(|y|\log(1 + \frac{1}{|y|}))^2}$ is integrable on $\D$, $A_2(g,\rho,x)$ tends uniformly to $0$ as $\rho \to 0$. Of course, $g$ is unbounded. To get an example of a bounded function, just consider $sin \circ g$.
\end{example}

With $\lambda > p-1$ and the involution $\varphi_x$ consider the mean oscillation in 
(\ref{Definition_rem_mean_oscillation}) again: 
\begin{equation*}
\textup{MO}^{\lambda}(f)(x)=
\int_{\mathbb{B}^n} \big{|}f \circ \varphi_x(y)- \mathcal{B}_{\lambda}(f)(x)\big{|}^2 dv_{\lambda}(y). 
\end{equation*}

\begin{lemma}\label{Lemma_estimate_mean_oscilation}
For all $x \in \mathbb{B}^n$ and all parameters $\rho > 0$, $\lambda > n$ we have
\begin{equation*}
\textup{MO}^{\lambda}(f)(x)\leq \int_{\mathbb{B}^n} \big{|} f \circ \varphi_x(y) - \hat{f}(x,\rho) \big{|}^2 dv_{\lambda}(y). 
\end{equation*} 
\end{lemma}

\begin{proof}
Let $x \in \mathbb{B}^n$ and fix the parameters $\rho > 0$ and $\lambda > p-1=n$. Consider the function $L: \mathbb{C} \rightarrow \mathbb{R}_+$ defined by 
\begin{equation*}
L(c):= \int_{\mathbb{B}^n} \big{|} f \circ \varphi_x(y) - c \big{|}^2 dv_{\lambda}(y).
\end{equation*}
Then the gradient $\textup{grad }L(c)$ vanishes precisely for 
\[c= \int_{\mathbb{B}^n}f \circ \varphi_x(y) dv_{\lambda}(y)= \mathcal{B}_{\lambda}(f)(x).\]
Since $L$ attains a minimum in the complex plane, the assertion follows. 
\end{proof}

\begin{theorem}\label{Theorem_VMO_inside_1}
Let $f \in  \textup{VMO}(\mathbb{B}^n)$, then $\lim_{\lambda \rightarrow \infty} \|f\|_{\textup{BMO}^{\lambda}}=0$. 
\end{theorem}

\begin{proof}
Let $t >0$ be a parameter which we will specify later on, $\lambda \geq p$ and put $\rho(\lambda):=c_{\lambda}^{-\frac{1}{2n}}$. The estimate in Lemma \ref{Lemma_estimate_mean_oscilation} shows that:
\begin{align}
\textup{MO}^{\lambda}(f)(x)
&\leq \Big{\{} \int_{E(0, t\rho(\lambda))}+ \int_{\beta(0,y) \geq t\rho(\lambda)}\Big{\}} \big{|} f \circ \varphi_x(y) - \hat{f}(x,t\rho(\lambda)) \big{|}^2 dv_{\lambda}(y)\notag\\
&=: I_{1, \lambda,t}(x)+I_{2, \lambda,t}(x). \label{Seond_integral_in_MO}
\end{align}
\par 
First we estimate the integral $I_{1, \lambda,t}(x)$ using the transformation rule together with the Cauchy-Schwarz inequality: 
\begin{multline*}
I_{1,\lambda,t}(x) = \int_{E(x, t\rho(\lambda))} \big{|}k_x^{\lambda}(y)\big{|}^2 \big{|} f(y)- \hat{f}(x, t\rho(\lambda)) \big{|}^2 dv_{\lambda}(y) \\
\leq c_{\lambda} \left\{\int_{E(x, t\rho(\lambda))} |k_x^{\lambda}(y)\big{|}^4 (1-|y|^2)^{2(\lambda-p)} dv(y) \right\}^{\frac{1}{2}}
\left\{ \int_{E(x, t\rho(\lambda))}\big{|} f(y)- \hat{f}(x, t\rho(\lambda)) \big{|}^4 dv(y) \right\}^{\frac{1}{2}}. 
\end{multline*}
We calculate the first integral on the right:
\begin{align}
\int_{E(x, t\rho(\lambda))} |k_x^{\lambda}(y)\big{|}^4& (1-|y|^2)^{2(\lambda-p)} dv(y) =\notag\\
=&\frac{1}{c_{2\lambda-p}} \int_{E(x,t\rho(\lambda))} |k_x^{2\lambda-p}(y)|^2 |k_x^p(y)|^2 dv_{2\lambda-p}(y)\notag \\
=& \frac{1}{c_{2\lambda-p}} \int_{E(0, t\rho(\lambda))} \big{|}k_x^p \circ \varphi_x(y)\big{|}^2 dv_{2\lambda-p}(y)=(+).
\label{estimate_integral_4th_power}
\end{align}
According to Proposition 3 in \cite{BCZ}, it holds 
\begin{equation*}
\big{|}k_x^p\circ \varphi_x(y) \big{|}^2 = \frac{1}{|k_x^p(y)|^2} = \frac{|h(y,x)|^{2p}}{h(x,x)^p}. 
\end{equation*} 
Therefore 
\begin{align*}
(+)&= \frac{1}{c_{2\lambda-p}h(x,x)^p} \int_{E(0, t\rho(\lambda))} |h(y,x)|^{2p} \,  dv_{2\lambda-p}(y) \leq \|h\|_{\infty}^{2p} \frac{\big{|} E(0,t \rho(\lambda))\big{|}}{h(x,x)^p},
\end{align*}
where $\|h\|_{\infty} = \sup\limits_{x,y \in \mathbb{B}^n} |h(y,x)| < \infty$. According to Lemma 1.23 in \cite{Zhu_book}, the volume of the Bergman ball $E(z,t\rho(\lambda))$ with $z \in \mathbb{B}^n$ 
is given by 
\begin{align}\label{Estimate_Bergman_volume}
|E(z,t\rho(\lambda))|
&=\frac{\tanh(t\rho(\lambda))^{2n} (1-|z|^2)^{n+1}}{(1- \tanh(t\rho(\lambda))^{2} |z|^2)^{n+1}} \\
&\leq C(t\rho(\lambda))^{2n} (1-|z|^2)^{n+1} = \frac{t^{2n}C}{c_{\lambda}} h(z,z)^p.\notag 
\end{align}
Here $C>0$ is a suitable constant and the estimate holds for large $\lambda > n$. Inserting this estimate above with $z=0$ gives 
\begin{equation*}
(+)\leq \frac{\|h\|_{\infty}^{2p} t^{2n}C}{c_{\lambda}} \frac{1}{h(x,x)^p}. 
\end{equation*}
Hence  we have 
\begin{equation*}
I_{1, \lambda,t}(x)\leq \|h\|_{\infty}^pt^n \sqrt{Cc_{\lambda}}\frac{|E(x, t\rho(\lambda))|^{\frac{1}{2}}}{h(x,x)^{\frac{p}{2}}} \left\{ \frac{1}{|E(x, t\rho(\lambda))|} \int_{E(x, t\rho(\lambda))} |f(y)- \hat{f}(x, t\rho(\lambda))|^4 dv(y)\right\}^{\frac{1}{2}}. 
\end{equation*}
By using (\ref{Estimate_Bergman_volume}) again one obtains as $\lambda \rightarrow \infty$:  
\begin{equation*}
I_{1, \lambda,t}(x) \leq \|h\|_{\infty}^pt^{2n}C \sqrt{A_4 \big{(}f, t\rho(\lambda),x\big{)}} \leq 2\|h\|_{\infty}^p t^{2n} C \|f\|_{\infty} \sqrt{A_2(f,t \rho(\lambda),x)}\rightarrow 0, 
\end{equation*}
where by the assumption on $f$ the above limit is uniform on $\mathbb{B}^n$. 
\vspace{1ex} \par 
Now we estimate the second integral in (\ref{Seond_integral_in_MO}) which we have denoted $I_{2, \lambda,t}(x)$. By \cite[Corollary 1.22]{Zhu_book}, we have $\beta(0,y) \geq t\rho(\lambda)$ 
if and only if $|y| \geq \tanh(t\rho(\lambda)) =: R_{\lambda}$. A change of variables shows that
\begin{align*}
c_{\lambda}\int_{\beta(0,y) \geq t \rho(\lambda)} (1-|y|^2)^{\lambda-p} \, dv(y) 
&= 2nc_{\lambda}\int_{R_{\lambda}}^1 (1-r^2)^{\lambda-n-1}r^{2n-1} \, dr\\
&= 2n \frac{c_{\lambda}}{\lambda^n} t^{2n} \int_{R_{\lambda} \frac{\sqrt{\lambda}}{t}}^{\frac{\sqrt{\lambda}}{t}}\Big{(} 1- \frac{t^2 r^2}{\lambda} \Big{)}^{\lambda-n-1} r^{2n-1} dr= g(R_{\lambda}). 
\end{align*}
We have the following asymptotic behavior as $\lambda \rightarrow \infty$: 
\begin{equation*}
R_{\lambda}^2 = \tanh(t \rho(\lambda))^2 \sim \frac{t^2}{c_{\lambda}^{1/n}} \sim \frac{t^2}{\lambda} \hspace{4ex} \mbox{\it and} \hspace{4ex} c_{\lambda} \sim \lambda^n. 
\end{equation*}
\par 
Moreover, for all $\lambda >p-1$ and fixed $t>0$ the integrand is dominated by $c\exp (-t^2r^2)r^{2n-1}$ where $c>0$ is a suitable constant independent of $t$ and $r$. Therefore the 
dominated convergence theorem implies: 
\begin{equation*}
\lim_{\lambda \rightarrow \infty} g(R_{\lambda})= 2n \gamma_1 t^{2n} \int_{\gamma_2}^{\infty} e^{-t^2r^2} r^{2n-1} dr =:H(t). 
\end{equation*}
\par 
Since $H(t) \rightarrow 0$ as $t \rightarrow \infty$ we can choose $t>0$ sufficiently large such that for $\lambda >M_1$ and 
all $x \in \mathbb{B}^n$ we have
\begin{equation*}
I_{2,\lambda,t}(x) \leq 4 \|f\|^2_{\infty} g(R_{\lambda}) \leq 8 \|f\|^2_{\infty}H(t) < \varepsilon. 
\end{equation*}
With this fixed $t$ we can choose $M_2>M_1$ such that $I_{1,\lambda,t}(x)\leq \varepsilon$ for $\lambda > M_2$ and all $x \in \mathbb{B}^n$. From (\ref{Seond_integral_in_MO}) we find 
$\textup{MO}^{\lambda}(f)(x)<2 \varepsilon$ uniformly on $\mathbb{B}^n$. 
\end{proof}

\begin{proposition}\label{proposition_hilf_1}
Let $\Omega$ be a BSD and $f \in \textup{BMO}^{\lambda}(\Omega)$. Then it holds
\[\sup\limits_{z \in \Omega} \mathcal{B}_{\lambda}(|f - \mathcal{B}_{\lambda}(f)|)(z) \leq C\|f\|_{\textup{BMO}^{\lambda}}\]
for some constant $C > 0$ (independent of $\lambda \geq 2p$).
\end{proposition}

\begin{proof}
By the triangular inequality, we have for all $z \in \Omega$: 
\begin{multline*}
\mathcal{B}_{\lambda}(|f - \mathcal{B}_{\lambda}(f)|)(z) = \int_{\Omega} |f(w) - \mathcal{B}_{\lambda}(f)(w)| |k_z^{\lambda}(w)|^2 \, dv_{\lambda}(w)\\
\leq \int_{\Omega} |f(w) - \mathcal{B}_{\lambda}(f)(z)||k_z^{\lambda}(w)|^2 \, dv_{\lambda}(w) 
 + \int_{\Omega} |\mathcal{B}_{\lambda}(f)(z) - \mathcal{B}_{\lambda}(f)(w)||k_z^{\lambda}(w)|^2 \, dv_{\lambda}(w).
\end{multline*}
According to the Cauchy-Schwarz inequality the first term is dominated by $\|f\|_{\textup{BMO}^{\lambda}}$. Hence it remains to estimate the second term. Using \cite[Theorem 4.9]{BC1} and 
Corollary \ref{Corollary_Schur_test_hilf_1}, we get
\begin{align*}
\int_{\Omega} |\mathcal{B}_{\lambda}(f)(z) - \mathcal{B}_{\lambda}(f)(w)|& |k_z^{\lambda}(w)|^2 \, dv_{\lambda}(w) 
\leq 2  \sqrt{\frac{\lambda}{p}}  \|f\|_{\textup{BMO}^{\lambda}} \int_{\Omega} \beta(z,w) |k_z^{\lambda}(w)|^2 \, dv_{\lambda}(w)\\
&=  2  \sqrt{\frac{\lambda}{p}}  \|f\|_{\textup{BMO}^{\lambda}}   \int_{\Omega} \beta(z,\varphi_z(w)) \, dv_{\lambda}(w)\\
&=  2  \sqrt{\frac{\lambda}{p}}  \|f\|_{\textup{BMO}^{\lambda}}   \int_{\Omega} \beta(0,w)  \, dv_{\lambda}(w)\\
&=  \frac{2C}{\sqrt{p}}\|f\|_{\textup{BMO}^{\lambda}} \int_{\Omega} h(w,w)^{-\frac{\lambda}{2}} \, dv_{\lambda}(w)\\
&= \frac{2C}{\sqrt{p}}\frac{c_{\lambda}}{c_{\frac{\lambda}{2}}}\|f\|_{\textup{BMO}^{\lambda}}. 
\end{align*}
\par 
Here $C>0$ is the constant in Corollary \ref{Corollary_Schur_test_hilf_1} which does not depend on $\lambda$. Since the quotient $\frac{c_{\lambda}}{c_{\frac{\lambda}{2}}}$ is bounded as a 
function of $\lambda$, the proposition follows.
\end{proof}
\begin{lemma}\label{Hilfslemma_Schur_test_2}
Let $\Omega$ be a BSD and $f \in L^{\infty}(\Omega)$. Then there is a constant $C > 0$ (independent of $\lambda \geq 2p$) such that
\begin{equation*}
 \|T^{\lambda}_{f}\|_{\lambda}\leq C \sqrt{\|f\|_{\infty} \| \mathcal{B}_{\frac{\lambda}{2}}(|f|)\|_{\infty}}, 
\end{equation*}
where we consider $T^{\lambda}_{f}=P_{\lambda}M_{f}$ as an operator acting on the whole space $L^2(\Omega, dv_{\lambda})$. 
\end{lemma}
\begin{proof}
We use the Schur test and put $h_{\lambda}(z) := h(z,z)^{-\frac{\lambda}{2}}$, $C := \sup\limits_{\lambda \geq 2p} \frac{c_{\lambda}}{c_{\frac{\lambda}{2}}}$. Recall that the Toeplitz operator 
$T_{f}^{\lambda}$ has the integral kernel 
\[T_{\lambda,f}(z,w) := f(z) h(z,w)^{-\lambda}.\]
Hence we obtain
\begin{align*}
\int_{\Omega} \big{|} T_{\lambda,f}(z,w)\big{|} h_{\lambda}(z) dv_{\lambda}(z)
&=c_{\lambda} \int_{\Omega} |f(z)| |h(z,w)|^{-\lambda} h(z,z)^{\frac{\lambda}{2}-p} dv(z)\\
&= \frac{c_{\lambda}}{c_{\frac{\lambda}{2}}} \int_{\Omega} |f(z)| |h(z,w)|^{-\frac{\lambda}{2} \cdot 2} dv_{\frac{\lambda}{2}}(z)\\
&= \frac{c_{\lambda}}{c_{\frac{\lambda}{2}}} h_{\lambda}(w) \mathcal{B}_{\frac{\lambda}{2}}(|f|)(w)\\
&\leq C_1 h_{\lambda}(w)
\end{align*}
with $C_1 = C\|\mathcal{B}_{\frac{\lambda}{2}}(|f|)\|_{\infty}$.
%\par
Integration with respect to the parameter $w$ yields: 
\begin{align*}
\int_{\Omega} \big{|} T_{\lambda,f}(z,w)\big{|} h_{\lambda}(w) dv_{\lambda}(w)
&= c_{\lambda}|f(z)| \int_{\Omega} |h(z,w)|^{-\lambda} h(w,w)^{\frac{\lambda}{2}-p} \, dv(w)\\
&=\frac{c_{\lambda}}{c_{\frac{\lambda}{2}}}|f(z)| \int_{\Omega} |h(z,w)|^{-\frac{\lambda}{2} \cdot 2} \, dv_{\frac{\lambda}{2}}(w)\\
&=\frac{c_{\lambda}}{c_{\frac{\lambda}{2}}}  |f(z)| h(z,z)^{-\frac{\lambda}{2}}\\
&\leq C_2h_{\lambda}(z)
\end{align*}
with $C_2 = C\|f\|_{\infty}$. Hence $\|T^{\lambda}_f\|_{\lambda}$ is dominated by $\sqrt{C_1C_2}=C\sqrt{\|f\|_{\infty} \| \mathcal{B}_{\frac{\lambda}{2}}(|f|)\|_{\infty}}$. 
\end{proof}

\begin{corollary}\label{corollary_estimate_norm_Hankel_operator}
Let $\Omega$ be a BSD, $f \in L^{\infty}(\Omega)$ and $C>0$ as above. Then 
\begin{equation}\label{norm_estimate_Hankel_operator_GL}
\|H_{f}^{\lambda} \|_{\lambda} \leq C\sqrt{\|f\|_{\infty} \| \mathcal{B}_{\frac{\lambda}{2}}(|f|)\|_{\infty}}. 
\end{equation}
\end{corollary}
\begin{proof}
Note that the norm of the multiplication $M_f: \mathcal{A}_{\lambda}^2(\Omega) \rightarrow L^2(\Omega, dv_{\lambda})$ coincides 
with the norm of 
\begin{equation*}
M_fP_{\lambda}= \big{[}P_{\lambda} M_{\overline{f}}\big{]}^* \in \mathcal{L}(L^2(\Omega, dv_{\lambda})). 
\end{equation*}
Therefore one has
\begin{align}\label{Estimate_norm_Hankel_opeator}
\|H_{f}^{\lambda} \|_{\lambda} 
&\leq  \| M_f: \mathcal{A}^2_{\lambda}(\Omega) \rightarrow L^2(\Omega, dv_{\lambda}) \|_{\lambda}
 = \|P_{\lambda} M_{\overline{f}} \|_{\lambda}
\end{align}
and Lemma \ref{Hilfslemma_Schur_test_2} implies the estimate (\ref{norm_estimate_Hankel_operator_GL}). 
\end{proof}

\begin{theorem}\label{theorem_quantization_VMO}
Let $f \in \textup{VMO}(\mathbb{B}^n)$. Then $\lim_{\lambda \rightarrow \infty} \|H_f^{\lambda}\|_{\lambda}=0$ and, in particular, we obtain
\begin{equation*}
\lim_{\lambda \rightarrow \infty} \|T_f^{\lambda} T_g^{\lambda} -T_{fg}^{\lambda} \|_{\lambda}=0
\end{equation*}
for all $g \in L^{\infty}(\mathbb{B}^n)$ or all $g \in \textup{UC}(\mathbb{B}^n)$. 
\end{theorem}

\begin{proof}
Because of (\ref{Semi_commutator_estimate}) and Theorem \ref{Theorem_behaviour_semi_commutator_uniformly_bounded_function} (in the case where  $g \in \textup{UC}(\mathbb{B}^n)$) 
it is sufficient to check that for all $f \in \textup{VMO}(\mathbb{B}^n)$: 
\begin{equation*}
\lim_{\lambda \rightarrow \infty} \| H_f^{\lambda}\|_{\lambda}=0. 
\end{equation*}
According to Corollary \ref{corollary_estimate_norm_Hankel_operator}, we have:
\begin{align}
\|H_f^{\lambda} \|_{\lambda} \label{Decomposition_Hankel_VMO}
 &\leq \big{\|} H_{f-\mathcal{B}_{\frac{\lambda}{2}}(f)}^{\lambda}\big{\|}_{\lambda} + \| H_{\mathcal{B}_{\frac{\lambda}{2}}(f)}^{\lambda} \big{\|}_{\lambda} \\
 &\leq C\sqrt{2 \|f\|_{\infty} \|\mathcal{B}_{\frac{\lambda}{2}}(|f -\mathcal{B}_{\frac{\lambda}{2}}(f)|) \|_{\infty}}+ \| H_{\mathcal{B}_{\frac{\lambda}{2}}(f)}^{\lambda} \big{\|}_{\lambda}. \notag
\end{align}
The first term on the right tends to zero as $\lambda \rightarrow \infty$ by the Proposition \ref{proposition_hilf_1} and Theorem \ref{Theorem_VMO_inside_1}.  
Note that according to (\ref{weighted_Bergman_metric}) we have $\sqrt{2} \beta_{\frac{\lambda}{2}}(z,w)= \beta_{\lambda}(z,w)$ and Theorem \ref{Theorem_Estimate_BMO_BO} implies that 
\begin{equation*}
\| \mathcal{B}_{\frac{\lambda}{2}}(f)\|_{\textup{BO}^{\lambda}} \leq \sqrt{2} \|f\|_{\textup{BMO}^{\frac{\lambda}{2}}}. 
\end{equation*}
The estimate (\ref{inequality_operator_norm_weighted_Hankel_operator}) implies that there is $C>0$ independent of $\lambda$ and $f$ with 
\begin{equation*}
 \| H_{\mathcal{B}_{\frac{\lambda}{2}}(f)}^{\lambda} \big{\|}_{\lambda}\leq C \|\mathcal{B}_{\frac{\lambda}{2}}(f) \| _{\textup{BO}^{\lambda}} \leq \sqrt{2} C \|f\|_{\textup{BMO}^{\frac{\lambda}{2}}}. 
\end{equation*}
\par 
Since $f \in \textup{VMO}(\mathbb{B}^n)$ it follows from Theorem \ref{Theorem_VMO_inside_1} that the right hand side tends to zero as $\lambda \rightarrow \infty$. Hence (\ref{Decomposition_Hankel_VMO}) proves the assertion. 
\end{proof}
%\begin{corollary}
%Let $f \in  \textup{VMO}(\mathbb{B}^n)$, then $\lim_{\lambda \rightarrow \infty} \|T_f^{\lambda}\|_{\lambda}= \|f\|_{\infty}$. 
%\end{corollary}
%\begin{proof}
%Follows from Theorem \ref{theorem_quantization_VMO} and Lemma \ref{lem_norm_convergence_TO}. 
%\end{proof} 
We add an observation on the asymptotic behavior of semi-commutators of Toeplitz operator and a relation to a compactness result in \cite{BBCZ}. With $n \in \mathbb{N}$ consider the standard monomial  orthonormal basis of 
$\mathcal{A}_{\lambda}^2(\mathbb{B}^{n+1})$ 
$$\mathcal{B}=\big{\{} e_{\alpha}^{\lambda}(z):=z^{\alpha} \|z^{\alpha}\|_{\lambda}^{-1} \: : \: \alpha \in \mathbb{Z}_+^{n+1}\big{\}}.$$ 
We split the coordinates $z \in \mathbb{B}^{n+1}$ and multi-indices $\alpha \in \mathbb{Z}_+^{n+1}$ into two parts: 
\begin{equation*}
z=(z^{\prime}, z^{\prime \prime}) \in \mathbb{C} \times \mathbb{C}^n \hspace{3ex} \mbox{\it and} \hspace{3ex} \alpha = (\alpha^{\prime}, \alpha^{\prime \prime}) \in \mathbb{Z}_+\times \mathbb{Z}_+^n.
\end{equation*}
A direct calculation (cf. \cite{BV}  in the case of $n=1$) shows that $e_{\alpha}^0(z)=e_{\alpha^{\prime}}^n(z^{\prime}) e_{\alpha^{\prime \prime}}^{|\alpha^{\prime}|+ 1}(z^{\prime \prime})$. This 
induces an orthogonal decomposition of the unweighted Bergman space 
\begin{equation}\label{unweighted_Bergman_space_orthogonal_decomposition}
\mathcal{A}^2_0(\mathbb{B}^{n+1})= \bigoplus_{j \in \mathbb{Z}_+} \underbrace{ \textup{span} \big{\{} e_{\alpha^{\prime}}^n(z^{\prime}) \: : \: \alpha^{\prime}=j\big{\}} \otimes \mathcal{A}_{j+1}^2(\mathbb{B}^n)}_{=:H_j}. 
\end{equation}
Let  $c,d \in L^{\infty}(\mathbb{B}^n)$  and extend $c$ to the ball $\mathbb{B}^{n+1}$ by $f_c(z):= c(z^{\prime \prime})$. The Toeplitz operator $T^0_{f_c}$ and the semi-commutator 
$T^0_{f_c}T^0_{f_d}-T^0_{f_cf_d}$ act on 
(\ref{unweighted_Bergman_space_orthogonal_decomposition}) as follows: 
\begin{align}
T_{f_c}^0
&= \bigoplus_{j=0}^{\infty} I \otimes T_c^{j+1}\notag \\
T^0_{f_c}T^0_{f_d}-T^0_{f_cf_d}
&=\bigoplus_{j=0}^{\infty} I \otimes \big{(} T^{j+1}_cT^{j+1}_d-T^{j+1}_{cd}\big{)}.\label{orthogonal_decomposition_semi_commutator}
\end{align} 
\par 
Assume that $f_c$ or $f_d$ are functions in $\textup{VMO}_{\partial}(\mathbb{B}^{n+1})$. Then it follows from the results in \cite{BBCZ} that the semi-commutator $T^0_{f_c}T^0_{f_d}-T^0_{f_cf_d}$ 
is compact on $\mathcal{A}_0^2(\mathbb{B}^{n+1})$. The decomposition (\ref{orthogonal_decomposition_semi_commutator}) and standard arguments imply that 
\begin{equation*}
T_c^{j+1}T_d^{j+1}-T_{cd}^{j+1} \in \mathcal{K}(\mathcal{A}_{j+1}^2(\mathbb{B}^n)) \hspace{3ex} \mbox{\it and} \hspace{3ex} \lim_{j \rightarrow \infty} \|T_c^{j+1}T_d^{j+1}-T_{cd}^{j+1}\|_j=0.
\end{equation*}
Introduce the $C^*$-algebra
\begin{equation*} 
\textup{VMO}^{\dag}(\mathbb{B}^n) := \Big{\{} c \in L^{\infty}(\mathbb{B}^n) \: : \: f_c \in \textup{VMO}_{\partial}(\mathbb{B}^{n+1}) \Big{\}}. 
\end{equation*}

Observe that if $z$ tends to $z_0 = (z'_0,z''_0) \in \partial \mathbb{B}^{n}$, with any $z'_0 \neq 0$, then the values of the function $f_c$ near the boundary point $z_0$ coincide with the values of the 
function $c$ in a neighborhood of the point $z''_0 \in \mathbb{B}^n$. Hence a ''vanishing oscillation condition'' of the function $f_c$ near the boundary may be interpreted as a ''vanishing oscillation condition'' 
of the function $c$ inside $\mathbb{B}^n$. This somewhat vague comment suggests the following 
\vspace{1mm}\\
{\sf conjecture:} $\textup{VMO}^{\dag}(\mathbb{B}^n) = \textup{VMO}(\mathbb{B}^n)$.

%%%%%%%%%%%%%%%%%%%%%%%%%%%%%%%%%%%%%%%%%%%%%%%%%%%%%%%%%%%%%%%%%%%%%%%%%%%%%%%%%%%%%%%%%%%%%%%%%%%

\end{document}